\documentclass[10pt]{amsart}
\usepackage{amsmath,amssymb,latexsym,esint,cite,mathrsfs}
\usepackage{verbatim,wasysym}
\usepackage[left=2.6cm,right=2.7cm,top=2.9cm,bottom=2.7cm]{geometry}
\usepackage{tikz,enumitem,graphicx, subfig, microtype, color}
\usepackage{epic,eepic}

\usepackage[colorlinks=true,urlcolor=blue, citecolor=red,linkcolor=blue,
linktocpage,pdfpagelabels, bookmarksnumbered,bookmarksopen]{hyperref}
\usepackage[hyperpageref]{backref}
\usepackage[english]{babel}

\numberwithin{equation}{section}

\newtheorem{thm}{Theorem}[section]
\newtheorem{lem}[thm]{Lemma}
\newtheorem{cor}[thm]{Corollary}
\newtheorem{Prop}[thm]{Proposition}

\newtheorem{Rem}[thm]{Remark}

\newcommand{\N}{\mathbb{N}}

\newcommand{\R}{\mathbb{R}}

\newcommand\cc{\mathcal{C}}

\def\de {\delta}

\def\ga {\gamma}

\def\vr {\varepsilon}
\def\va {\varphi}

\begin{document}

\title[Critical Choquard equations]{Semiclassical states for Choquard type equations with critical growth: critical frequency case}

\author[Y.\ Ding]{Yanheng Ding}
\author[F.\ Gao]{Fashun Gao}
\author[M.\ Yang]{Minbo Yang$^*$}

\address[Y.\ Ding]
{Institute of Mathematics,\newline\indent
 Academy of Mathematics and Systems Sciences\newline\indent
Chinese Academy of Sciences,\newline\indent
 Beijing 100080, P. R. China.}
\email{dingyh@math.ac.cn}

\address[F.\ Gao]{Department of Mathematics,\newline\indent
	Zhejiang Normal University, \newline\indent
	321004, Jinhua, Zhejiang, P. R. China}
\email{fsgao@zjnu.edu.cn}

\address[M.\ Yang]{Department of Mathematics,\newline\indent
	Zhejiang Normal University, \newline\indent
	321004, Jinhua, Zhejiang, P. R. China}
\email{mbyang@zjnu.edu.cn}


\thanks{$^*$Minbo Yang is the corresponding author who is partially supported by NSFC(11571317,11671364); Yanheng Ding is partially supported by NSFC(11331010,11571146).}

\begin{abstract}
	In this paper we are interested in the existence of semiclassical states for the Choquard type equation
	$$
	-\vr^2\Delta u +V(x)u =\Big(\int_{\R^N} \frac{G(u(y))}{|x-y|^\mu}dy\Big)g(u)  \quad \mbox{in $\R^N$},
	$$
	where $0<\mu<N$, $N\geq3$, $\vr$ is a positive parameter and $G$ is the primitive of $g$ which is of critical growth due to the Hardy--Littlewood--Sobolev inequality. The potential function $V(x)$ is assumed to be nonnegative with $V(x)=0$ in some region of $\R^N$, which means it is of the critical frequency case. Firstly, we study a Choquard equation with double critical exponents and prove the existence and multiplicity of semiclassical states by the Mountain-Pass Lemma and the genus theory. Secondly, we consider a class of critical Choquard equation without lower perturbation, by establishing a global Compactness lemma for the nonlocal Choquard equation, we prove the multiplicity of high energy semiclassical states by the Lusternik--Schnirelman theory.
\end{abstract}

\maketitle

\begin{center}
	\begin{minipage}{8.5cm}
		\small
		\tableofcontents
	\end{minipage}
\end{center}
%

\section{Introduction and main results}
In this paper we are interested in the existence and multiplicity of semiclassical states for
the following nonlocal semilinear equation
\begin{equation}\label{SNS}
 -\vr^2\Delta u +V(x)u  =\Big(\int_{\R^N}  K(x-y)G(u(y))dy\Big)g(u).
\end{equation}
This type of equation is closely related to the nonlocal evolutional Schr\"{o}dinger equation
\begin{equation}\label{ESNS}
i\hbar\partial_{t} \Psi =-\frac{\hbar^2}{2m} \Delta \Psi
+W(x)\Psi-\Big(\int_{\R^N} K(x-y)G(\Psi)dy\Big)g(\Psi) ,\ \ \
x\in \R^N,
\end{equation}
where $m$ is the mass of the bosons, $\hbar$ is the planck constant, $G$ is the primitive of $g$, $W(x)$ is the external potential and $K(x)$ is the
function which possesses information on the mutual interaction
between the bosons. It is clear that
$\Psi(x,t)=u(x)e^{\frac{-iEt}{\hbar}}$  solves \eqref{ESNS} if and only if $u(x)$ solves
 equation \eqref{SNS}
with $V(x)=W(x)-E$ and $\vr^2=\frac{\hbar^2}{2m}$.

If the
response function $K(x)=\delta(x)$ the impulsive function, the
nonlinear response is local indeed, then the nonlinear equation \eqref{SNS} goes back to the classical local Schr\"{o}dinger equation
 $$
 -\vr^2\Delta u +V(x)u  =g(u)  \quad \mbox{in $\R^N$}.
$$
 As $\vr$ goes to zero, the study of the existence and asymptotic behavior of the solutions is known as the {\em semiclassical problem} which was used to describe the transition between
Quantum Mechanics and Classical Mechanics. In mathematical aspects the study of semiclassical problem goes back to the pioneer work \cite{FW} by Floer and Weinstein. Since then, it has been studied
extensively under various hypotheses on the potentials and the nonlinearities, see for example
 \cite{ABC, AMSS, BW, CL, CP, DF2, DL, DW, DF1, FM, JT, KW, R, WX, WZ, ZCZ} and
the references therein.

However nonlocality appears naturally in optical systems with a thermal
\cite{L} and it is known to influence the propagation of electromagnetic
waves in plasmas \cite{BC}.  Nonlocality also has
attracted considerable interest as a means of eliminating
collapse and stabilizing multidimensional solitary waves \cite{Ba} and it plays an important
role in the theory of Bose-Einstein condensation \cite{D} where it accounts for the finite-range many-body interactions. If $K(x)$ is a function of Riesz type $K(x)=\frac{1}{|x|^\mu}$, then we arrive at the singularly perturbed Choquard type equation\begin{equation}\label{SCPP}
-\vr^2\Delta u +V(x)u =\Big(\int_{\R^N} \frac{G(u(y))}{|x-y|^\mu}dy\Big)g(u)  \qquad \mbox{in $\R^N$},
\end{equation}
where $N\geq3$,  $0<\mu<N$. For  $\vr=1$, $N=3$, $g(u)=u$ and $\mu=1$, the equation
\begin{equation}\label{AP}
 -\Delta u +u =\Big(\int_{\R^3} \frac{|u(y)|^2}{|x-y|}dy\Big)u   \qquad \mbox{in $\R^3$}
\end{equation}
was introduced in mathematical physics by Pekar \cite{P1} to study the quantum theory of a polaron at rest. It was mentioned in \cite{Li1} that Choquard applied it as approximation to Hartree-Fock theory of one-component plasma. This equation was also proposed by Penrose in \cite{MPT} as a model of selfgravitating matter and is known in that context as the Schr\"odinger-Newton equation. Mathematically, Lieb \cite{Li1} and
Lions \cite{Ls} studied the existence  and uniqueness of
positive solutions to equation \eqref{AP}. The uniqueness and non-degeneracy of the ground states were proved in Lenzmann\cite{Len}, Wei
and Winter in \cite{WW}.

To study problem \eqref{SCPP} variationally, we will use the following Hardy--Littlewood--Sobolev inequality frequently, see \cite{LL}.
\begin{Prop}\label{HLS}
 (Hardy--Littlewood--Sobolev inequality).  Let $t,r>1$ and $0<\mu<N$ with $1/t+\mu/N+1/r=2$, $f\in L^{t}(\R^N)$ and $h\in L^{r}(\R^N)$. There exists a sharp constant $C(t,N,\mu,r)$, independent of $f,h$, such that
\begin{equation}\label{HLS1}
\int_{\mathbb{R}^{N}}\int_{\mathbb{R}^{N}}\frac{f(x)h(y)}{|x-y|^{\mu}}dxdy\leq C(t,N,\mu,r) |f|_{t}|h|_{r},
\end{equation}
where $|\cdot|_{q}$ for the $L^{q}(\mathbb{R}^{N})$-norm for $q\in[1,\infty]$. If $t=r=2N/(2N-\mu)$, then
$$
 C(t,N,\mu,r)=C(N,\mu)=\pi^{\frac{\mu}{2}}\frac{\Gamma(\frac{N}{2}-\frac{\mu}{2})}{\Gamma(N-\frac{\mu}{2})}\left\{\frac{\Gamma(\frac{N}{2})}{\Gamma(N)}\right\}^{-1+\frac{\mu}{N}}.
$$
In this case there is equality in \eqref{HLS1} if and only if $f\equiv Ch$ and
$$
h(x)=A(\gamma^{2}+|x-a|^{2})^{-(2N-\mu)/2}
$$
for some $A\in \mathbb{C}$, $0\neq\gamma\in\mathbb{R}$ and $a\in \mathbb{R}^{N}$.
\end{Prop}
Let $H^{1}(\R^N)$ be the usual Sobolev spaces with the norm
    \[
    \|u\|_{H^1}:=\left(\int_{\R^N}(|\nabla u|^2+ |u|^2)dx\right)^{1/2},
    \]
    $D^{1,2}(\R^N)$ be equipped with the norm
    $$
\|u\|:=\Big(\int_{\R^N}|\nabla u|^{2}dx\Big)^{\frac{1}{2}}
$$
and $L^s(\R^N)$, $1 \leq s \leq \infty$,
denotes the Lebesgue space with the norms
\begin{gather*}
| u |_s:=\Big(\int_{\R^N}|u|^sdx\Big)^{1/s}.
\end{gather*}
By the Hardy--Littlewood--Sobolev inequality and the Sobolev imbedding, the integral
$$
\int_{\mathbb{R}^{N}}\int_{\R^N}\frac{|u(x)|^{t}|u(y)|^{t}}{|x-y|^{\mu}}dxdy
$$
is well defined if
$$
\frac{2N-\mu}{N}\leq t\leq\frac{2N-\mu}{N-2}.
$$
Let
\begin{equation}\label{CES}
2_{\mu \ast}:=\frac{2N-\mu}{N}\ \hbox{and} \ 2_{\mu}^{\ast}:=\frac{2N-\mu}{N-2},
\end{equation}
in the rest of this paper we will call the exponent $2_{\mu \ast}$ the lower critical exponent, while $2_{\mu}^{\ast}$ the upper critical exponent. Recently, by using
the method of moving planes,  Ma and Zhao \cite{ML} studied
the classification of all positive solutions to the generalized
nonlinear Choquard problem
\begin{eqnarray}\label{limit equation3}
-\Delta u+u=\Big(\int_{\R^N} \frac{|u(y)|^p}{|x-y|^\mu}dy\Big)|u|^{p-2}u,
\end{eqnarray}
under some assumptions on $\mu$, $p$ and $N$,  they proved that
all the positive solutions of \eqref{limit equation3} must be radially
symmetric and monotone decreasing about some fixed point.
In \cite{MS1}, Moroz and Van
Schaftingen completely investigated the qualitative properties of solutions of \eqref{limit equation3} and showed the regularity, positivity
and radial symmetry decay behavior at infinity.
For autonomous equation
$$
-\Delta u+u =\left(\int_{\R^N}\frac{|u(y)|^{p}}{|x-y|^{\mu}}dy\right)|u|^{p-2}u \quad \mbox{in} \quad \R^N
$$
with $p=2_{\mu}^{\ast}$ or $p=2_{\mu \ast}$, one can follow the steps in \cite{GY, MS1} to establish the Poho\v{z}aev identity
$$
\frac{N-2}{2}\int_{\R^N} |\nabla u|^{2}dx+\frac{N}{2}\int_{\R^N} |u|^{2}dx=\frac{2N-\mu}{2p}\int_{\R^N}
\int_{\R^N}\frac{|u(x)|^{p}|u(y)|^{p}}{|x-y|^{\mu}}dxdy,
$$
then it is easy to see that there are no nontrivial solutions. Because the problem was set in $\R^N$ and the convolution type nonlinearities are of critical growth, it is quite difficult to study the critical Choquard equation \eqref{limit equation3} due to the loose of compact embedding. For the upper critical exponent case, a critical Choquard type equation on a bounded domain of $\R^N$, $N\geq 3$  was investigated in \cite{GY}, there the authors generalized the well-known results obtained in \cite{BN}. The critical Choquard equation set on the whole space was investigated in \cite{ACTY, AGSY}, where the authors studied the case $N=2$ with exponential critical growth and the case $N\geq 3$ with upper critical exponents separately.

The appearance of potential well function $V(x)$ influences the existence of solutions greatly. Consider
\begin{eqnarray}\label{PC1}
-\Delta u+V(x)u=\Big(\int_{\R^N} \frac{|u(y)|^p}{|x-y|^\mu}dy\Big)|u|^{p-2}u,
\end{eqnarray}
as we all know if the potential $V(x)$ is a continuous periodic function, the spectrum of the Schr\"{o}dinger operator $-\Delta +V$ is purely continuous and consists of a union of closed intervals. If $\inf_{\R^3} V(x)> 0$ and $
\frac{2N-\mu}{N}\leq p<\frac{2N-\mu}{N-2}
$, since the energy functional is invariant under translation, the existence of ground states by applying the Mountain Pass Theorem, see \cite{AC} for example. For the critical growth case, if $N\geq3$, the existence of ground states were obtained in \cite{AGSY} by applying the Brezis-Nirenberg methods.
The planar case was considered in \cite{ACTY}, where the authors first established the existence of ground state for the problem with critical exponential growth.
If $V(x)$ changes sign, the operator $-\Delta +V$ has essential spectrum below $0$ and then equation \eqref{CCE1} becomes strongly indefinite.
In contrast to the positive definite case, it becomes more complicated to study the strongly indefinite Choquard equation due to the appearance of convolution part. For $p=2$ and $\mu=1$, the existence of one nontrivial solution was obtained in \cite{BJS} by reduction arguments. For a general class of subcritical Choquard type equation
 \begin{equation}\label{WCh}
 -\Delta u +V(x)u =\Big(\int_{\R^N}W(x-y)|u(y)|^{p}dy\Big)|u|^{p-2}u  \quad \mbox{in} \quad \R^3,
\end{equation}
the existence of solutions was obtained in \cite{AC} by applying a generalized linking theorem, where $W(x)>0$ belongs to a wide class of functions. The author also proved the existence of infinitely many geometrically distinct weak solutions. For $N\geq4$ and $\mu$ lies in suitable range, Gao and Yang \cite{GY2} considered the strongly indefinite case and obtain the existence of nontrivial solution by applying the generalized linking theorem. For the Choquard equation with lower critical exponent, Moroz and Van Schaftingen \cite{MS4} studied the existence of solutions by perturbing the linear part suitably. For recent progress on the study of the Choquard equation we may refer the readers to \cite{ANY, guide} for details.

The semiclassical problem for the Choquard equation has also attracted a lot interest recently. As far as we know there are some papers considered the equation of the type
 \begin{equation}\label{SCP}
-\vr^2\Delta u +V(x)u =\vr^{\mu-N}\Big(\int_{\R^N} \frac{Q(y)G(u(y))}{|x-y|^\mu}dy\Big)Q(x)g(u)  \qquad \mbox{in $\R^N$}.
\end{equation}
 It can be observed that if $u$ is a solution of the nonlocal equation (\ref{SCP}), for $x_0\in\R^N$, the function $v=u(x_0+\vr x)$ satisfies
 \begin{equation}\label{EP1}
-\Delta v +V(x_0+\vr x)v  =\Big(\int_{\R^N} \frac{Q(x_0+\vr y)G(v(y))}{|x-y|^\mu}dy\Big)Q(x_0+\vr x)g(v) \qquad \mbox{in $\R^N$}.
\end{equation}
 It suggests some convergence, as $\vr\to0$, of the family of solutions to a solution $u_0$ of the limit problem
\begin{equation*}
-\Delta v +V(x_0)v=Q^2(x_0)\Big(\int_{\R^N} \frac{G(v(y))}{|x-y|^\mu}dy\Big)g(v) \qquad \mbox{in $\R^N$}.
\end{equation*}
Hence we know that the equation
\begin{equation}\label{CC}
-\Delta u +u  =\Big(\int_{\R^N} \frac{G(v(y))}{|x-y|^\mu}dy\Big)g(u) \qquad \mbox{in $\R^{3}$}
\end{equation}
plays the role of limit equation in the study of the semiclassical problems for Choquard equation. To apply the Lyapunov-Schmidt reduction techniques, it relies a lot on the uniqueness and non-degeneracy of the ground states of the limit problem which is not completely known for the nonlocal Choquard equation \eqref{CC}.
If $\inf V>0$, Wei
and Winter also constructed families of solutions
 by a Lyapunov-Schmidt type reduction  for
  \begin{equation}\label{SCP0}
-\vr^2\Delta u +V(x)u =\vr^{-2}\Big(\int_{\R^3} \frac{|u(y)|^2}{|x-y|}dy\Big)u,  \qquad \mbox{in $\R^3$}.
\end{equation}
 Cingolani et.al. \cite{CCS} applied the penalization arguments due to Byeon and Jeanjean \cite{BJ} and showed that there exists a family of solutions
having multiple concentration regions which are located around the minimum points of the potential. The result in \cite{CCS} was recently generalized to nonlinearities of Berestycki-Lions type by Yang, Zhang and Zhang \cite{YZZ}, the authors also established the existence of multi-peak solutions.  For any $N\geq 3$ and $G(u)=u^p$ with $2_{\mu\ast}\leq p<2_{\mu}^{\ast}$ in \eqref{SCP}, Moroz and Van Schaftingen \cite{MS3} developed a nonlocal penalization technique and showed that equation
\begin{equation}\label{SCP1}
-\vr^2\Delta u +V(x)u =\vr^{\mu-N}\Big(\int_{\R^N} \frac{|u(y)|^p}{|x-y|}dy\Big)|u|^{p-2}u,  \qquad \mbox{in $\R^N$}
\end{equation}
has a family of solutions concentrating around the local minimum of $V$ with $V$ satisfying some additional assumptions at infinity. In \cite{AY1, AY2}, by applying penalization method and Lusternik-Schnirelmann theory, Alves and Yang proved the existence, multiplicity and concentration of solutions for the equation \eqref{SCP} with subcritical nonlinearities. For the critical growth case, the authors in \cite{AGSY} studied
	$$
	-\vr^2\Delta u +V(x)u =\vr^{\mu-3}\Big(\int_{\R^3} \frac{Q(y)G(u(y))}{|x-y|^\mu}dy\Big)Q(x)g(u)  \quad \mbox{in $\R^3$},
	$$
	where $0<\mu<3$, $\vr$ is a positive parameter, $V,Q$ are two continuous real function on
	$\R^3$. Assuming that $g$ which is of upper critical growth, the authors  established the existence and multiplicity of semiclassical states and then characterized the concentration behavior around the the global minimum set of $V$ or the global maximum set of $Q$.
The planar case was investigated in \cite{ACTY}, there the authors observed the concentration of the maximum points of the solutions around the global minimum set of the potential.

 In the present paper we continue to study the semiclassical problem for the Choquard equation, but instead of problem \eqref{SCP}, we are going to study problem \eqref{SCPP}.
We must point out that problem \eqref{SCPP} is quite different from \eqref{SCP}, since there is no small parameter in front of the nonlinear convolution part. For cubic type nonlinearities $G(u)=u^p$ with $2_{\mu\ast}\leq p<2_{\mu}^{\ast}$, by setting $v(x)=\vr^{\frac{\mu-N}{2(p-1)}}u(x_0+ \vr x)$, equation \eqref{SCPP} can still be transformed into an equivalent form \eqref{EP1}. However, this scaling transformation does not apply for  equation \eqref{SCPP} with general nonlinearities, even for $g(u)=u^p+u^q,\ p\neq q$. Involving the existence of semiclassical states of equation \eqref{SCPP}, there are not so many results.
 If the nonlinearity is subcritical and the potential $V(x)$ is of critical frequency, i.e. $V(x)\geq0, \min V(x)=0$, Yang and Ding \cite{YD} considered the equation
 $$
-\vr^2\Delta u +V(x)u  =\Big(\int_{\R^3} \frac{u^p(y)}{|x-y|^\mu}dy\Big)u^{p-1} \quad \mbox{in $\R^3$}
$$
with critical frequency $\inf V(x)=0$ and $2_{\mu\ast}\leq p<2_{\mu}^{\ast}$, as a particular case the authors proved the existence of solutions for lower critical exponent case.
This type of problem has also been considered by  Van Schaftingen and Xia \cite{SX}, the authors characterized the concentration behavior of the solutions for $2_{\mu\ast}< p<2_{\mu}^{\ast}$ and $V(x)$ satisfies some more assumptions. Cingolani and Secchi \cite{CS} studied the semiclassical limit for the pseudo-relativistic Hartree equation
 $$
\sqrt{-\vr^2\Delta +m^2 }u +V(x)u  =\Big(\int_{\R^N} \frac{u^p(y)}{|x-y|^\mu}dy\Big)u^{p-1} \quad \mbox{in $\R^N$},
$$
by using the local realization by means of Dirichlet-to-Neumann operator, they were able to establish the existence of single-spike solution concentrating around the local minimum set of $V(x)$ by penalization techniques.

From the comments above, we know that the existing results for problem \eqref{SCPP} with critical frequency $\inf V(x)=0$ are all about the subcritical case. And there seems no existence and multiplicity results for problem \eqref{SCPP} with both of the upper critical growth and critical frequency. Then it is quite natural to ask if we can establish existence and multiplicity of standing waves with critical frequency for the Choquard equation with upper critical growth. The aim of the present paper is to give a positive answer to the this question  and we are going to study how the behavior of the nonnegative potentials and the upper critical exponent will affect the existence and multiplicity of semiclassical states of problem \eqref{SCPP}.

Firstly we are going to study the existence, multiplicity of semiclassical states of the Choquard equation with double critical exponents, that is
\begin{equation}\label{SCP2}
-\vr^2\Delta u+V(x)u =\Big(\int_{\R^N}\frac{|u(y)|^{2_{\mu}^{\ast}}}{|x-y|^{\mu}}dy\Big)|u|^{2_{\mu}^{\ast}-2}u +\Big(\int_{\R^N}\frac{|u(y)|^{2_{\mu\ast}}}{|x-y|^{\mu}}dy\Big)|u|^{2_{\mu \ast}-2}u  \quad \mbox{in} \quad \R^N.
\end{equation}
Under some flatness assumption on the potential, we find that the combination of the upper critical exponent and the lower critical exponent will still lead to the existence and multiplicity of semiclassical states with small energy depending on the parameter $\vr$.

In order to state the main results, we assume that $0<\mu<N$ and the potential $V(x)$ satisfy
\begin{itemize}
\item[$(V_1)$] $V\in \mathcal{C}(\R^N)$ and
there is $b>0$ such that the set $\mathcal{V}^b:=\{x\in\R^N: V(x)<b\}$ has
finite Lebesgue measure.
 \item[$(V_2)$] $0=V(0)\leq V(x),\ x\in \R^N$.
 \item[$(V_3)$] There exists $0<\tau\leq\frac{1}{N-\mu+2}$ such that
$$
 \displaystyle\lim_{|x|\to0}\frac{V(x)}{|x|^{\frac{1-2\tau}{\tau}}}=0.$$
 \end{itemize}
 \begin{Rem}
 In fact the assumption $(V_3)$ implies that
 $$
 V(x)=o(|x|^\vartheta),\ \  \hbox{as}\ \ |x|\to0, \ \ \vartheta\in [N-\mu, +\infty).
 $$ 
 \end{Rem}
Under the assumptions above we can state the existence result as follow:
\begin{thm}\label{1}
Let conditions $(V_1)-(V_3)$ be satisfied. Then for any
$\de>0$ there is $\mathcal{E}_\de>0$ such that if
$\vr\leq\mathcal{E}_\de$, equation \eqref{SCP2} has at least one ground state solution $u_\vr$ satisfying
$$
\int_{\R^N}\Big(\vr^2|\nabla u_\vr|^2+ V(x)|u_\vr|^2\Big)dx\leq  \de2^{\frac{2N-\mu}{N-\mu}}\vr^{\frac{2(1-2\tau)(2N-\mu)}{N-\mu}}.
$$
Furthermore, there exists constant $C_0$ such that
$$
\|u_\vr\|^2_{H^1}\leq  \de C_0\vr^{\frac{2N(1-4\tau)+4\mu\tau}{N-\mu}},
$$
since $0<\tau\leq\frac{1}{N-\mu+2}$, then $2N(1-4\tau)+4\mu\tau>0$ which means that $u_\vr\to 0$ in $H^1(\R^N)$ as $\vr\to 0$.
\end{thm}
We also have the multiplicity result for the critical Choquard equation.
\begin{thm}\label{2}
Let conditions $(V_1)-(V_3)$ be satisfied. Then,
for any $m\in\N$ and $\de>0$, there is $\mathcal{E}_{m\de}>0$ such
that \eqref{SCP2} has at least $m$ pairs of solutions $u_\vr$ satisfying
$$
\int_{\R^N}\Big(\vr^2|\nabla u_\vr|^2+ V(x)|u_\vr|^2\Big)dx\leq  m\de2^{\frac{2N-\mu}{N-\mu}}\vr^{\frac{2(1-2\tau)(2N-\mu)}{N-\mu}}.
$$
Furthermore, there exists constant $C_0$ such that
$$
\|u_\vr\|^2_{H^1}\leq  \de C_0\vr^{\frac{2N(1-4\tau)+4\mu\tau}{N-\mu}},
$$
which means that $u_\vr\to 0$ in $H^1(\R^N)$ as $\vr\to 0$.
\end{thm}

Secondly, we are going to study the nonlinear critical Choquard equation without lower perturbation. Consider
\begin{equation}\label{CE2}
\displaystyle-\vr^{2}\Delta u+ V(x)u
=\Big(\int_{\R^N}\frac{|u(y)|^{2_{\mu}^{\ast}}}{|x-y|^{\mu}}dy\Big)|u|^{2_{\mu}^{\ast}-2}u
\hspace{4.14mm}\mbox{in}\hspace{1.14mm} \R^N,
\end{equation}
where the potential $V$ satisfies the assumptions
\begin{itemize}
\item[$(V_4)$] $V\geq0$ on $\R^N$ and the set $M=\{x\in \R^N:V(x)=0\}$ is nonempty and
bounded.
 \item[$(V_5)$] $\exists p_{1}<\frac{N}{2},p_{2}>\frac{N}{2}$ and for $N=3$, $p_{2}<3$, such that
$$
V(x)\in L^{p},\ \ \ \forall p\in[p_{1},p_{2}].
$$
 \end{itemize}

Recall that if $Y$ is a closed set of a topological space $X$, we denote by $cat_X(Y)$, the Ljusternik-
Schnirelmann category of $Y$ in $X$, namely the least number of closed and contractible
sets in $X$ which cover $Y$. For $\tau>0$ small, let
$$
M_{\tau}:=\{x\in \R^N:dist(x, M)\leq\tau\},
$$
then the multiplicity of solutions for the nonlocal problem can be characterized by the Lusternik-Schnirelman category $cat_{M_{\tau}}M$.
\begin{thm}\label{3}
Suppose that conditions $(V_4)$ and $(V_5)$ hold, $0<\mu<\min\{4,N\}$ and $N\geq3$. Then for $\vr$ small problem \eqref{CE2} has $cat_{M_{\tau}}M$ solutions.
\end{thm}

For a given functional $I\in \mathcal{C}^1(E,\R)$,
$\{u_n\}\subset E$ is said to be Palais-Smale sequence at $c$ for $I$ ($(PS)_c$ sequence for short) if $I(u_n)\to c$ and
$I'(u_n)\to0$ as $n\to\infty$. We say that $I$
satisfies $(PS)_c$ condition if any
$(PS)_c$ sequence has a convergent
subsequence. In this paper
 we use $C$, $C_i$ to denote positive constants and $B_R$ the open ball centered at the origin with
radius $R>0$. $\mathcal{C}_0^{\infty}(\R^N)$ denotes  functions
infinitely differentiable with compact support in $\R^N$. The best Sobolev constant $S$ is defined by:
\[
S|u|^2_{2^*}\leq \int_{\R^N}|\nabla u|^2dx \ \ \ \hbox{for all $u\in
D^{1,2}(\R^N)$}.
\]

As \cite{GY}, let $S_{H,L}$ be the best constant defined by
\begin{equation}\label{S1}
S_{H,L}:=\displaystyle\inf\limits_{u\in D^{1,2}(\R^N)\backslash\{{0}\}}\ \ \frac{\displaystyle\int_{\R^N}|\nabla u|^{2}dx}{\Big(\displaystyle\int_{\R^N}\int_{\R^N}
\frac{|u(x)|^{2_{\mu}^{\ast}}|u(y)|^{2_{\mu}^{\ast}}}{|x-y|^{\mu}}dxdy\Big)^{\frac{N-2}{2N-\mu}}}.
\end{equation}

\begin{lem}\label{ExFu} (\cite{GY})
The constant $S_{H,L}$ defined in \eqref{S1} is achieved if and only if $$u=C\left(\frac{b}{b^{2}+|x-a|^{2}}\right)^{\frac{N-2}{2}} ,$$ where $C>0$ is a fixed constant, $a\in \R^N$ and $b\in(0,\infty)$ are parameters. What's more,
$$
S_{H,L}=\frac{S}{C(N,\mu)^{\frac{N-2}{2N-\mu}}},
$$
where $S$ is the best Sobolev constant.
\end{lem}

An outline of this paper is as follows: In Section 2, we study the Choquard equation with double critical exponents and prove the existence of semiclassical states for equation \eqref{SCP2} by energy estimates and Mountain-Pass Theorem. We also proved that equation \eqref{SCP2} has at least $m$ pairs of solutions by the Krasnoselski genus theory. In Section 3, we prove firstly a global compactness lemma for the nonlocal Choquard equation and establish a convergence criteria for the $(PS)$ sequences. And then we prove the existence of multiple high energy semiclassical solutions of problem \eqref{CE2} by the Lusternik--Schnirelman theory.

\section{Critical problem with double critical exponents}
To prove the existence of semiclassical states by variational methods, we introduce the Hilbert spaces
\[
E:=\left\{u\in H^1(\R^N): \, \int_{\R^N}V(x)u^2dx<\infty\right\}
\]
with the inner products
\[
(u, v):=\int_{\R^N}\big(\nabla u\nabla v+V(x)uv\big)dx
\]
and the associated norms
\[\|u\|^2_V=(u,u).
\]
 Obviously, it follows from $(V_1)$
that $E$ embeds continuously in $H^1(\R^N)$ (see \cite{DL,Si}).  Note that the norm
$\|\cdot\|_V$ is equivalent to $\|\cdot\|_{\vr}$ deduced by
the inner product
\[
(u, v)_{\vr}:=\int_{\R^N}\big(\vr^2\nabla u\nabla v +  V(x)uv\big)dx
\]
for each $\vr>0$.

Consider the Choquard equation with double critical exponents
\begin{equation}\label{SCC1}
-\vr^2\Delta u + V(x)u  =\left(\int_{\R^N}\frac{|u(y)|^{2_{\mu}^{\ast}}}{|x-y|^{\mu}}dy\right)|u|^{2_{\mu}^{\ast}-2}u +\left(\int_{\R^N}\frac{|u(y)|^{2_{\mu\ast}}}{|x-y|^{\mu}}dy\right)|u|^{2_{\mu \ast}-2}u \,\,\,\, \mbox{in $\R^{N}$},
\end{equation}
we can define the functional on $E$ by
$$
I_{\vr}(u)=\frac12\|u\|^2_{\vr}-\frac{1}{2\cdot2_{\mu}^{\ast}}\int_{\R^N}\int_{\R^N}
\frac{|u(x)|^{2_{\mu}^{\ast}}|u(y)|^{2_{\mu}^{\ast}}}{|x-y|^{\mu}}dxdy-\frac{1}{2\cdot2_{\mu\ast}}\int_{\R^N}\int_{\R^N}
\frac{|u(x)|^{2_{\mu \ast}}|u(y)|^{2_{\mu \ast}}}{|x-y|^{\mu}}dxdy.
$$
The Hardy--Littlewood--Sobolev inequality implies that $I_{\vr}$ is well defined on $E$ and belongs to $\mathcal{C}^{1}$. And so $u$ is a weak solution of \eqref{SCC1} if and only if $u$ is a critical point of the functional $I_{\vr}$.

\subsection{Existence of ground states}

We will use the following Mountain--Pass Theorem to prove the existence of solutions.
\begin{lem}\label{mountain:1} \cite{AR} Let $E$ be a real Banach space and $I:E\to \R$ a functional of class $\mathcal{C}^1$.  Suppose  that $I(0)=0$ and:
\begin{itemize}
  \item[$(I_1).$] There exist $\kappa, \rho>0$ such that $I|_{\partial B_\rho}\geq\kappa>0$ for all $u\in \partial B_\rho=\{u\in E:\|u\|=\rho\}$;
  \item[$(I_2).$]  there is $e$ with $\|e\|>\rho$ such that $I(e)\leq 0$.
\end{itemize}
Then I possesses a $(PS)_c$ sequence with $c\geq\kappa>0$ given by
\[
c:=\inf_{\gamma \in \Gamma} \max_{0\leq t\leq 1}
I(\gamma (t)),
\]
where
\[
\Gamma = \{\gamma \in \mathcal{C}([0,1],E):
\gamma(0)=0,\quad
\gamma (1)=e\}. \quad
\]
\end{lem}
Generally, we need to verify that the functional $I_\vr$ satisfies the Mountain--Pass Geometry. In fact,
\begin{lem}\label{MPgeometry}
Let conditions $(V_1)-(V_3)$ be satisfied, then  for each $\vr>0$, $I_{\vr}(0)=0$ and there
exists $\rho_\vr>0$ such that
$\kappa_\vr:=\inf I_\vr(\partial B_{\rho_\vr})>0$ where
$\partial B_{\rho_\vr}=\{u\in E: \ \|u\|_\vr=\rho_\vr\}$.
\end{lem}
\begin{proof}
First, for each fixed $\vr$,  $I_{\vr}(0)=0$.
Applying the Hardy--Littlewood--Sobolev inequality, for each $u\in E$,
 we know
\begin{equation}\label{mpg1}
\aligned
I_{\vr}(u)\geq\frac12\|u\|^2_{\vr}-C_0\|u\|^{2\cdot2_{\mu\ast}}_{\vr}-C_1\|u\|_{\vr}^{2\cdot2_{\mu}^{\ast}}
,
\endaligned
\end{equation}
the conclusion follows if $\|u\|_{\vr}$ is small enough.

Moreover, for any $u_{1}\in E\backslash\ \{0\}$, we have
$$
I_{\vr}(tu_{1})\leq\frac{t^{2}}{2}\|u_{1}\|^2_{\vr}-\frac{t^{2\cdot2_{\mu}^{\ast}}}{2\cdot2_{\mu}^{\ast}}
\int_{\mathbb{R}^N}\int_{\mathbb{R}^N}
\frac{|u_{1}(x)|^{2_{\mu}^{\ast}}|u_{1}(y)|^{2_{\mu}^{\ast}}}
{|x-y|^{\mu}}dxdy<0
$$
for $t>0$ large enough.
\end{proof}
The following Proposition is taken from \cite{YD},
\begin{Prop}\label{pro2}
\begin{equation}\label{Infimum}
\inf\,\left\{\int_{\R^N}|\nabla\varphi|^2dx: \
\varphi\in\mathbb{C}^\infty_0(\R^N),\,  \int_{\R^N}\int_{\R^N}\frac{|\varphi(x)|^{2_{\mu\ast}}|\varphi(y)|^{2_{\mu\ast}}}{|x-y|^{\mu}}dxdy=1\right\} = 0.\end{equation}
  \end{Prop}
\begin{proof}
In fact, for all fixed $\varphi$ satisfying
$$
 \int_{\R^N}\int_{\R^N}\frac{|\varphi(x)|^{2_{\mu\ast}}|\varphi(y)|^{2_{\mu\ast}}}{|x-y|^{\mu}}dxdy=1,
 $$
let us define,
 $$
\varphi_t=t^{\frac{N}{2}}\varphi(tx),\ \ t>0.
 $$
Then we have
 $$
 \int_{\R^N}\int_{\R^N}\frac{|\varphi_t(x)|^{2_{\mu\ast}}|\varphi_t(y)|^{2_{\mu\ast}}}{|x-y|^{\mu}}dxdy=\int_{\R^N}\int_{\R^N}\frac{|\varphi(x)|^{2_{\mu\ast}}|\varphi(y)|^{2_{\mu\ast}}}{|x-y|^{\mu}}dxdy=1
 $$
 and
 $$
 \int_{\R^N}|\nabla\varphi_t|^2dx=t^{2} \int_{\R^N}|\nabla\varphi|^2dx.
 $$
Thus we know
 $$
 \int_{\R^N}|\nabla\varphi_t|^2dx\to0
 $$
 as $t\to0$. Thus the proposition is proved.
\end{proof}
 \begin{lem}\label{AF}
Let conditions $(V_1)-(V_3)$ be satisfied. Then for any
$\de>0$ there exists $\mathcal{E}_\de>0$ such that, for each
$\vr\leq\mathcal{E}_\de$, there is $\psi_{\vr}\in E$ such that $I_{\vr}(\psi_{\vr})< 0$ and $$
 \max_{t\in [0,1]}I_{\vr}(t\psi_{\vr})\leq\frac{N-\mu}{2N-\mu}2^{\frac{N}{N-\mu}}\vr^{\frac{2(1-2\tau)(2N-\mu)}{N-\mu}}.
$$
\end{lem}
 \begin{proof}
From Proposition \ref{pro2}, for any $\de>0$ one can choose $\va_\de\in\cc^\infty_0(\R^N)$ with $Supp\va_\de\subset B_{r_\de}(0)$ such that
$$\int_{\R^N}\int_{\R^N}\frac{|\va_\de(x)|^{2_{\mu \ast}}|\va_\de(y)|^{2_{\mu \ast}}}
{|x-y|^{\mu}}dxdy=1$$ and
$$|\nabla\va_\de|^2_2<\de.$$
For $\tau$ in condition $(V_3)$, define
\begin{equation}\label{mpg2}
\psi_\vr(x):=\vr^{-\frac{N}{2}}\va_\de(\vr^{-2\tau}x),
\end{equation}
then
$$
Supp \psi_\vr \subset B_{\vr^{2\tau}r_\de}(0).
$$
It is easy to see that
$$
\int_{\R^N}|\nabla \psi_\vr|^2dx=\vr^{(2\tau-1)N-4\tau}\int_{\R^N}|\nabla
\va_\de|^2dx,
$$
$$
\int_{\R^N}V(x)|\psi_\vr|^2dx=\vr^{(2\tau-1)N}\int_{\R^N}
V(\vr^{2\tau}x)|\va_\de(x)|^2dx
$$
and
$$
\int_{\R^N}\int_{\R^N}\frac{|\psi_\vr(x)|^{2_{\mu \ast}}|\psi_\vr(y)|^{2_{\mu \ast}}}
{|x-y|^{\mu}}dxdy=\vr^{(2\tau-1)(2N-\mu)}.
$$
From $Supp\va_\de\subset B_{r_\de}(0)$ and the fact that $$
 \displaystyle\lim_{|x|\to0}\frac{V(x)}{|x|^{\frac{1-2\tau}{\tau}}}=0,$$
we know that there is $\mathcal{E}_{\de,1}>0$ such that for any $0<\vr<\mathcal{E}_{\de,1}$
$$
V(\vr^{2\tau}x)\leq \frac{\vr^{2(1-2\tau)}\de}{|\va_\de|^2_2}
$$
uniformly for $x\in B_{r_\de}(0)$.
Then from the above equalities, we know
\begin{equation}\label{ES1}
\aligned
 \displaystyle I_{\vr}(\psi_{\vr})
 &\leq\displaystyle \frac12\int_{\R^N}\vr^2\big|\nabla\psi_{\vr}\big|^2dx+\frac12\int_{\R^N}V(x)|\psi_{\vr}|^2dx-\frac{1 }{2\cdot2_{\mu \ast}}\int_{\R^N}\int_{\R^N}\frac{|\psi_\vr(x)|^{2_{\mu \ast}}|\psi_\vr(y)|^{2_{\mu \ast}}}
{|x-y|^{\mu}}dxdy\vspace{5mm}\\
&\displaystyle={\frac{\vr^{(2\tau-1)(N-2)}}{2}}\int_{\R^N}|\nabla
\va_\de|^2dx+\frac{\vr^{(2\tau-1)N}}{2}\int_{\R^N}
V(\vr^{2\tau}x)\va^2_\de dx
-\frac{\vr^{(2\tau-1)(2N-\mu)}}{2\cdot2_{\mu \ast}}\\
&<\de\vr^{(2\tau-1)(N-2)}-\frac{\vr^{(2\tau-1)(2N-\mu)}}{2\cdot2_{\mu \ast}}.
\endaligned
\end{equation}
Since $0<\tau\leq\frac{1}{N-\mu+2}$, then $$
(2\tau-1)(2N-\mu)<(2\tau-1)(N-2),
  $$ thus we know there exists $\mathcal{E}_{\de}$ with $0<\mathcal{E}_{\de}<\mathcal{E}_{\de,1}$ such that, for any $0<\vr<\mathcal{E}_{\de}$
there is a $\psi_{\vr}$ such that
$$
 I_{\vr}(\psi_{\vr})<0.
 $$
 Observe that  $I_{\vr}(t\psi_{\vr})>0$ for $t$ small enough and $I_{\vr}(t\psi_{\vr})<0$ for $t\geq1$, we know $$
 \max_{t\in \R}I_{\vr}(t\psi_{\vr})=\max_{t\in [0,1]}I_{\vr}(t\psi_{\vr}).
$$
Moreover, for such fixed $0<\vr<\mathcal{E}_{\de}$, from \eqref{ES1} we know
 $$
\aligned
\displaystyle\max_{t\in [0,1]}I_{\vr}(t\psi_{\vr})&\leq \de\max_{t\in [0,1]}\{\vr^{(2\tau-1)(N-2)}t^2-\frac{\vr^{(2\tau-1)(2N-\mu)}}{2\cdot2_{\mu \ast}}t^{2\cdot2_{\mu \ast}}\}\\
&:=\de\max_{t\in [0,1]}\Psi(t).
\endaligned
$$
By direct computation, we know there exists unique $t_0\in [0,1]$ such that
$$
\Psi(t_0)=\max_{t\in [0,1]}\Psi(t).
$$
In fact, $t_0$ satisfies
$$
t^2_0=2^{\frac{N}{N-\mu}}\vr^{\frac{(N+2-\mu)(1-2\tau)N}{N-\mu}}.
$$
Consequently, we know
$$
\aligned
\max_{t\in [0,1]}I_{\vr}(t\psi_{\vr})\leq \de\frac{N-\mu}{2N-\mu}2^{\frac{N}{N-\mu}}\vr^{\frac{2(1-2\tau)(2N-\mu)}{N-\mu}}
\endaligned
$$
and the conclusion is proved.
\end{proof}

For $\de>0$ and $\mathcal{E}_\de>0$ obtained in Lemma \ref{AF}, $I_\vr$ possesses a $(PS)_{c_\vr}$ sequence with $c_\vr\geq\kappa_\vr>0$ given by
$$
c_\vr:=\inf_{\ga\in\Gamma_\vr}\max_{t\in [0,1]}I_\vr(\ga(t))
$$
where
$$
\Gamma_\vr:=\left\{\ga\in\cc([0,1],E): \ \ga(0)=0 \ \hbox{and} \
\ga(1)= \psi_{\vr}\right\}.
$$
By Lemma \ref{AF}, we know
$$
0<\kappa_\vr\leq c_\vr\leq\de\frac{N-\mu}{2N-\mu}2^{\frac{N}{N-\mu}}\vr^{\frac{2(1-2\tau)(2N-\mu)}{N-\mu}}.
$$

\begin{lem}\label{PSlemma:1}
Suppose that conditions $(V_1)-(V_3)$ hold. For fixed $0<\vr<\mathcal{E}_{\de}$ small,
  let
$\{u_n\}$ be a $(PS)_{c_\vr}$ sequence for $I_{\vr}$, then $\{u_n\}$ is bounded.
\end{lem}
\begin{proof}
Let $\{u_n\}$ be a $(PS)_{c_\vr}$ sequence, i.e. $\{u_n\}$ satisfies that
$$
I_\vr(u_n)\to c_\vr \ \ \ \hbox{and} \ \ \
 I'_\vr(u_n)\to 0.
$$
We see that
$$
\aligned
c_\vr&+o_n(1)\|u_n\|_\vr\\
&=
I_\vr(u_n)-\frac{1}{2\cdot2_{\mu \ast}}(I'_\vr(u_n), u_n) \\ &=\frac{N-\mu}{2(2N-\mu)}\|u_n\|^2_\vr+\frac{1}{2N-\mu}\int_{\R^N}\int_{\R^N}
\frac{|u(x)|^{2_{\mu}^{\ast}}|u(y)|^{2_{\mu}^{\ast}}}{|x-y|^{\mu}}dxdy\\
&\geq\frac{N-\mu}{2(2N-\mu)}\|u_n\|^2_\vr,
\endaligned
$$
which means $\{u_n\}$ is bounded.
\end{proof}

Hence, without loss of generality, we may assume that $u_n\rightharpoonup u$ in $E$ and $L^2(\R^N)$, $u_n\to u$ in
$L^s_{loc}(\R^N)$ for $1\leq s< 2^*$, and $u_n(x)\to u(x)$ a.e. for
$x\in\R^N$. Clearly $u$ is a critical point of $I_{\vr}$. Denote
$$
\|u\|_{NL}:=\left(\int_{\mathbb{R}^N}\int_{\mathbb{R}^N}\frac{|u(x)|^{2_{\mu}^{\ast}}|u(y)|^{2_{\mu}^{\ast}}}
{|x-y|^{\mu}}dxdy\right)^{\frac{1}{2\cdot2_{\mu}^{\ast}}},
$$
the following splitting Lemma was proved in \cite{GY}.
\begin{lem} \label{BLN}Let $N\geq3$ and $0<\mu<N$. If $\{u_{n}\}$ is a bounded sequence in $L^{\frac{2N}{N-2}}(\mathbb{R}^N)$ such that $u_{n}\rightarrow u$ almost everywhere in $\mathbb{R}^N$ as $n\rightarrow\infty$, then the following hold,
$$
\|u_{n}\|_{NL}^{2\cdot2_{\mu}^{\ast}}
-\|u_{n}-u\|_{NL}^{2\cdot2_{\mu}^{\ast}}\rightarrow\|u\|_{NL}^{2\cdot2_{\mu}^{\ast}}
$$
as $n\rightarrow\infty$.
\end{lem}

\begin{lem}\label{PSlemma:2}
Suppose that conditions $(V_1)-(V_3)$ hold. For $0<\vr<\mathcal{E}_{\de}$ small, let
$\{u_n\}$ be a $(PS)_{c_\vr}$ sequence for $I_{\vr}$. One has along a subsequence in Lemma \ref{PSlemma:1}:
\begin{itemize}
\item[$(1).$] $I_{\vr}(u_n-u)\to c_\vr-I_{\vr}(u)\geq0$;
\item[$(2).$] $I'_{\vr}(u_n-u)\to 0$.
\end{itemize}
\end{lem}
\begin{proof}
Notice that $u$ is a critical point of $I_{\vr}$, by the nonlocal Brezis-Lieb type lemma \ref{BLN}, we see that $\{u_n-u\}$ is a  $(PS)_{c_\vr-I_{\vr}(u)}$ sequence for $I_{\vr}$. From the arguments in Lemma \ref{PSlemma:1}, we also have $c_\vr-I_{\vr}(u)\geq0$.
\end{proof}
Next we denote by
$$
w_{n}:=u_{n}-u,
$$
 then $u_{n}\to u$ in $E$ if and only if $w_n\to 0$ in $E$.

\begin{lem}\label{T0}
Suppose that conditions $(V_1)-(V_3)$ hold.  Then for any $\eta>0$ there exists $\mathcal{E}_\eta$ such that, for any $\vr\leq\mathcal{E}_\eta$ there holds
$$
\int_{\R^N}|w_{n}|^2dx\leq\eta
$$
for $n$ large enough.
\end{lem}
\begin{proof}
Since the Hilbert space $E$ embeds continuously in $H^1(\R^N)$ (see \cite{DL,Si}), we know there exists a constant $C>0$ independent of $\vr$ such that
$$
\int_{\mathbb{R}^N}(|\nabla u|^{2}+|u|^{2})dx\leq C\int_{\mathbb{R}^N}(|\nabla u|^{2}+ V(x)|u|^{2})dx.
$$
Hence, from the proof of Lemma \ref{PSlemma:1}, we know that $\{w_{n}\}$ is bounded and satisfies
$$
\vr^2\int_{\mathbb{R}^N}(|\nabla w_{n}|^{2}+|w_{n}|^{2})dx\leq C_1\big(c_\vr-I_{\vr}(u)+o_n(1)\big).
$$
Consequently, by $I_{\vr}(u)\geq 0$, we know
$$
\int_{\mathbb{R}^N}|w_{n}|^{2}dx\leq C_1\frac{\big(c_\vr-I_{\vr}(u)+o_n(1)\big)}{\vr^2}\leq  C_1\frac{c_\vr+o_n(1)}{\vr^2}.
$$
Recall that, for any $\de$ there exists $\mathcal{E}_{\de}$, such that, for $\vr<\mathcal{E}_{\de}$ there holds
$$
c_\vr\leq\de\frac{N-\mu}{2N-\mu}2^{\frac{N}{N-\mu}}\vr^{\frac{2(1-2\tau)(2N-\mu)}{N-\mu}}.
$$
Since $0<\tau\leq\frac{1}{N-\mu+2}$, we know
$$
\frac{2(1-2\tau)(2N-\mu)}{N-\mu}-2=\frac{2N(1-4\tau)+4\mu\tau}{N-\mu}>0,
$$
then the Mountain Pass value satisfies
$$
c_\vr=o(\vr^2),\ \ \hbox{as}\ \  \vr\to0.
$$
Thus, for any $\eta>0$, there exists $\mathcal{E}_\eta<\mathcal{E}_{\de}$ such that
$$
\int_{\mathbb{R}^N}|w_{n}|^{2}dx\leq C_1\frac{\big(c_\vr-I_{\vr}(u)+o_n(1)\big)}{\vr^2}< \eta +\frac{o_n(1)}{\vr^2}
$$
for any $\vr\leq\mathcal{E}_\eta$. Consequently, for such $\vr$, we know
$$
\int_{\mathbb{R}^N}|w_{n}|^{2}dx\leq \eta
$$
for $n$ large enough.
\end{proof}

\begin{lem}\label{PSlemma:3}
Suppose that conditions $(V_1)-(V_3)$ hold. For $\vr>0$ small enough, let
$\{u_n\}$ be a $(PS)_{c_\vr}$ sequence for $I_{\vr}$ with
$$
c_{\vr}< \frac{N-\mu+2}{2(2N-\mu)}S_{H,L}^{\frac{2N-\mu}{N-\mu+2}}\vr^{\frac{2(2N-\mu)}{N-\mu+2}},
$$
then it contains a convergent subsequence.
\end{lem}
\begin{proof}
We need only to check that the
$(PS)_{c_\vr}$ sequence $\{u_n\}$ contains a strongly convergent subsequence.
Let
$$
V_{b}(x):=\max\{V(x), b\},
$$
where $b$ is the positive constant from assumption $(V_1)$. Since the set $\mathcal{V}^b$ has finite measure and $w_{n}\to0$ in $L^2_{loc}$, we see that
$$
\int_{\R^N} V(x)|u_n|^2dx=\int_{\R^N} V_{b}(x)|u_n|^2dx+o_n(1).
$$
From Lemma \ref{PSlemma:2},  along a subsequence, we have
\begin{center}
 $I_{\vr}(w_n)\to c_\vr-I_{\vr}(u)$\ \
and \ \  $I'_{\vr}(w_n)\to 0$.
\end{center}
Thus
$$
I_{\vr}(w_n)-\frac12(I'_{\vr}(w_n),w_{n})\geq\frac{N-\mu+2}{2(2N-\mu)}
\int_{\R^N}\int_{\R^N}\frac{|w_n(x)|^{2_{\mu}^{\ast}}|w_n(y)|^{2_{\mu}^{\ast}}}{|x-y|^{\mu}}dxdy,
$$
i.e.
\begin{equation}\label{ECT}
\int_{\R^N}\int_{\R^N}\frac{|w_n(x)|^{2_{\mu}^{\ast}}|w_n(y)|^{2_{\mu}^{\ast}}}{|x-y|^{\mu}}dxdy\leq \frac{2(2N-\mu)}{N-\mu+2}\big(c_{\vr}-I_{\vr}(u)\big)+o_n(1).
\end{equation}
On the other hand, by the definition of $S_{H,L}$ we know
$$\aligned
\vr^2S_{H,L}&\Big(\displaystyle\int_{\R^N}\int_{\R^N}
\frac{|w_n(x)|^{2_{\mu}^{\ast}}|w_n(y)|^{2_{\mu}^{\ast}}}{|x-y|^{\mu}}dxdy
\Big)^{\frac{N-2}{2N-\mu}}\\
&\leq \int_{\R^N}\Big(\vr^2|\nabla w_n|^2+ V(x)|w_n|^2\Big)dx-\int_{\R^N} V(x)|w_n|^2dx\\
&\leq C(N,\mu)\Big(\int_{\R^N}|w_n|^2dx\Big)^{\frac{2N-\mu}{N}}+\int_{\R^N}\int_{\R^N}
\frac{|w_n(x)|^{2_{\mu}^{\ast}}|w_n(y)|^{2_{\mu}^{\ast}}}{|x-y|^{\mu}}dxdy-\int_{\R^N} V_{b}(x)|w_n|^2dx+o_n(1)\\
&\leq  C(N,\mu)\Big(\int_{\R^N}|w_n|^2dx\Big)^{\frac{2N-\mu}{N}}+\int_{\R^N}\int_{\R^N}
\frac{|w_n(x)|^{2_{\mu}^{\ast}}|w_n(y)|^{2_{\mu}^{\ast}}}{|x-y|^{\mu}}dxdy-b\int_{\R^N} |w_n|^2dx+o_n(1)\\
&=C(N,\mu)\Big(\Big(\int_{\R^N}|w_n|^2dx\Big)^{\frac{N-\mu}{N}}-\frac{b}{C(N,\mu)}\Big)\int_{\R^N} |w_n|^2dx +\int_{\R^N}\int_{\R^N}\frac{|w_n(x)|^{2_{\mu}^{\ast}}|w_n(y)|^{2_{\mu}^{\ast}}}
{|x-y|^{\mu}}dxdy+o_n(1).
\endaligned
$$
Choosing
$$\eta=\Big(\frac{b}{C(N,\mu)}\Big)^{\frac{N}{N-\mu}},$$
by Lemma \ref{T0} there exists $\mathcal{E}_\eta>0$ such that, for any $\vr\leq\mathcal{E}_\eta$ there holds
$$
\int_{\R^N}|w_{n}|^2dx\leq\eta
$$
for $n$ large enough. Therefore
$$\aligned
\vr^2S_{H,L}&\Big(\displaystyle\int_{\R^N}\int_{\R^N}
\frac{|w_n(x)|^{2_{\mu}^{\ast}}|w_n(y)|^{2_{\mu}^{\ast}}}{|x-y|^{\mu}}dxdy\Big)^{\frac{N-2}{2N-\mu}}\leq\int_{\R^N}\int_{\R^N}\frac{|w_n(x)|^{2_{\mu}^{\ast}}|w_n(y)|^{2_{\mu}^{\ast}}}{|x-y|^{\mu}}dxdy+o_n(1).
\endaligned
$$

Assume now that $\{u_n\}$ has no convergent subsequence, then
$\liminf_{n\to\infty}\|w_n\|_{\vr}>0$ and $c-I_{\vr}(u)>0$. Thus we can get
$$\aligned
\vr^2S_{H,L}\leq\Big(\displaystyle\int_{\R^N}\int_{\R^N}
\frac{|w_n(x)|^{2_{\mu}^{\ast}}|w_n(y)|^{2_{\mu}^{\ast}}}{|x-y|^{\mu}}dxdy\Big)^{\frac{N-\mu+2}{2N-\mu}}+o_n(1).
\endaligned
$$
By \eqref{ECT}, we know
$$\aligned
\vr^2S_{H,L}\leq\Big(\frac{2(2N-\mu)}{N-\mu+2}\Big)^{\frac{N-\mu+2}{2N-\mu}}\Big(c_{\vr}-I_{\vr}(u)\Big)^{\frac{N-\mu+2}{2N-\mu}}+o_n(1),
\endaligned
$$
which means
$$\aligned
c_{\vr}\geq \frac{N-\mu+2}{2(2N-\mu)}S_{H,L}^{\frac{2N-\mu}{N-\mu+2}}\vr^{\frac{2(2N-\mu)}{N-\mu+2}},
\endaligned
$$
this is a contradiction.
\end{proof}

{\bf Proof of Theorem \ref{1}.}  For $\de>0$,  from Lemma \ref{AF} and Lemma \ref{PSlemma:3}, we know there exists $\mathcal{E}_\de>0$ such that, for any $\vr<\mathcal{E}_\de$ the functional $I_\vr$ possesses a $(PS)_{c_\vr}$ sequence $\{u_n\}$ with
$$
c_\vr\leq\de\frac{N-\mu}{2N-\mu}2^{\frac{N}{N-\mu}}\vr^{\frac{2(1-2\tau)(2N-\mu)}{N-\mu}}.
$$
Since $0<\tau\leq\frac{1}{N-\mu+2}$, we have
$$
\frac{2(1-2\tau)(2N-\mu)}{N-\mu}\geq \frac{2(2N-\mu)}{N-\mu+2},
$$
thus if $\de$ is small enough, we know
$$
\de\frac{N-\mu}{2N-\mu}2^{\frac{N}{N-\mu}}\vr^{\frac{2(1-2\tau)(2N-\mu)}{N-\mu}}\leq \frac{N-\mu+2}{2(2N-\mu)}S_{H,L}^{\frac{2N-\mu}{N-\mu+2}}\vr^{\frac{2(2N-\mu)}{N-\mu+2}}.
$$
Applying Lemma \ref{PSlemma:3}, we know that $\{u_n\}$
contains a convergent subsequence. The Mountain Pass Theorem implies that
there is $u_\vr\in E$ such that $I'_\vr(u_\vr)=0$ and
$I_\vr(u_\vr)=c_\vr$. Moreover, one can see that
$$
\|u_\vr\|^2_\vr\leq \de2^{\frac{2N-\mu}{N-\mu}}\vr^{\frac{2(1-2\tau)(2N-\mu)}{N-\mu}}.
$$
Furthermore, since $0<\tau\leq\frac{1}{N-\mu+2}$, we know
$$
\frac{2(1-2\tau)(2N-\mu)}{N-\mu}-2=\frac{2N(1-4\tau)+4\mu\tau}{N-\mu}>0,
$$
then we have
$$
\int_{\R^N}\Big(|\nabla u_\vr|^2+ V(x)|u_\vr|^2\Big)dx\leq  \de2^{\frac{2N-\mu}{N-\mu}}\vr^{\frac{2N(1-4\tau)+4\mu\tau}{N-\mu}},
$$
which means that $u_\vr$ goes to $0$, as $\vr\to0$.

\subsection{Multiple semiclassical states}
In order to obtain the multiplicity of critical points, we will apply the index theory defined by the Krasnoselski genus.
Denote the set of all symmetric (in the sense that $-A=A$) and
closed subsets of $E$ by $\Sigma$. For each $A\in\Sigma$, let
$gen(A)$ be the Krasnoselski genus and
$$
i(A) := \min_{h\in\Gamma} gen(h(A)\cap \partial B_{\vr}),
$$
where $\Gamma$ is the set of all odd homeomorphisms $h\in
\cc(E,E)$ and $\partial B_{\vr}$ is the closed symmetric set
$$
\partial B_{\vr}:=\Big\{v\in E:\|v\|_\vr=\rho_\vr\Big\}
$$
 such that $I_\vr |_{\partial B_\vr}\geq\kappa_\vr  > 0$.
Then $i$ is a version of Benci's pseudoindex
\cite{B}. Let
\begin{equation}\label{index}
c_{\vr j} := \inf_{i(A)\ge j}\ \sup_{u\in A}I_\vr(u), \quad
1\le j\le m.
\end{equation}
Then if $c_{\vr j}$ is finite and $I_\vr$ satisfies the $(PS)$ condition at $c_{\vr j}$, then we know $c_{\vr j}$ are all critical values for $I_\vr$.

\begin{lem}\label{MPgeometry1}
Let conditions $(V_1)-(V_3)$ be satisfied, then the functional $I_{\vr}$ satisfies:
\begin{itemize}
  \item[$(1).$] for each $\vr>0$, $I_{\vr}(0)=0$  there
exists $\rho_\vr>0$ such that
$\kappa_\vr:=\inf I_\vr(\partial B_{\rho_\vr})>0$ where
$\partial B_{\rho_\vr}=\{u\in E: \ \|u\|_\vr=\rho_\vr\}$;
     \item[$(2).$] for each $\vr>0$ and
any finite-dimensional subspace $F\subset E$, there is
$R=R(\vr,F)>0$ such that $I_\vr(u)\leq 0$ for all $u\in F$ with
$\| u\|_{\vr}\geq R$.
\end{itemize}
\end{lem}
\begin{proof} $(1).$ It is proved in Lemma \ref{MPgeometry}.

$(2)$.
Define
$$
\nu:=\inf_{w\in F, \|w\|_\vr=1}\Big\{\int_{\R^N}\int_{\R^N}\frac{|w(x)|^{2_{\mu \ast}}|w(y)|^{2_{\mu \ast}}}{|x-y|^{\mu}}dxdy\Big\}.
$$
Since $F$ is a finite-dimensional subspace of $E$, we must have $\nu>0$.
Therefore
$$
\aligned
I_\vr(u)&\leq\frac12\|u\|^2_{\vr}-\frac{1}{2\cdot2_{\mu\ast}}\int_{\R^N}\int_{\R^N}\frac{|u(x)|^{2_{\mu \ast}}|u(y)|^{2_{\mu \ast}}}
{|x-y|^{\mu}}dxdy\\
&=\frac12\|u\|^2_{\vr}-\frac{\|u\|^{2\cdot2_{\mu \ast}}_\vr}{2\cdot2_{\mu \ast}}\int_{\R^N}\int_{\R^N}\frac{|u(x)|^{2_{\mu \ast}}|u(y)|^{2_{\mu \ast}}}
{\|u\|^{2\cdot2_{\mu \ast}}_\vr|x-y|^{\mu}}dxdy,\\
&\leq\frac12\|u\|^2_{\vr}-\frac{\nu}{2\cdot2_{\mu\ast}}\|u\|^{2\cdot2_{\mu \ast}}_\vr.
\endaligned
$$
Consequently, we know
$$
I_\vr(u)\to-\infty \ \ \ \hbox{as
$\|u\|_\vr\to\infty$} .
$$
\end{proof}

\begin{proof}[Proof of Theorem \ref{2}] From lemma \ref{MPgeometry1}, we know for each $\vr$ there is a closed subset $\partial B_{\rho_\vr}$ of $E$ and $\kappa_\vr > 0$ such that the even functional $I_\vr |_{\partial B_{\rho_\vr}}\geq\kappa_\vr  > 0$.

For any $m\in\N$, one can choose $m$ functions
$\va^j_\de\in\cc^\infty_0(\R^N)$ such that
$Supp\va^i_\de\cap Supp\va^k_\de=\emptyset$ if $i\not=k$,
$$\int_{\R^N}\int_{\R^N}\frac{|\va^j_\de(x)|^{2_{\mu \ast}}|\va^j_\de(y)|^{2_{\mu \ast}}}
{|x-y|^{\mu}}dxdy=1$$  and $$|\nabla\va^j_\de|^2_2<\de.$$ Let
$r^m_\de>0$ be such that $Supp\va^j_\de\subset B_{r^m_\de}(0)$
for $j=1,...,m$. For $\tau$ in assumption $(V_3)$, set
\begin{equation}\label{mpg2}
\psi^j_{\vr}(x):=\psi_\vr(x):=\vr^{-\frac{N}{2}}\va^j_\de(\vr^{-2\tau}x),  s\in \R
\end{equation}
and
$$
H^m_{\vr\de}=Span\{\psi^1_\vr, ..., \psi^m_\vr\}.
$$
For each $u=\sum^m_{j=1}t_j\psi^j_\vr\in H^m_{\vr\de}$, it is easy to see that
$$
\int_{\R^N}|\nabla u|^2dx=\vr^{(2\tau-1)N-4\tau}\sum^m_{j=1}|t_j|^2\int_{\R^N}|\nabla
\va^j_\de|^2dx,
$$
$$
\int_{\R^N}V(x)|u|^2dx=\vr^{(2\tau-1)N}\sum^m_{j=1}|t_j|^2\int_{\R^N}
V(\vr^{2\tau}x)|\va^j_\de(x)|^2dx
$$
and
$$\aligned
\int_{\R^N}\int_{\R^N}\frac{|u(x)|^{2_{\mu \ast}}|u(y)|^{2_{\mu \ast}}}
{|x-y|^{\mu}}dxdy&\geq \sum^m_{j=1}|t_j|^{2\cdot2_{\mu \ast}}\int_{\R^N}\int_{\R^N}\frac{|\va^j_\de(x)|^{2_{\mu \ast}}|\va^j_\de(y)|^{2_{\mu \ast}}}
{|x-y|^{\mu}}dxdy\\
&=m\vr^{(2\tau-1)(2N-\mu)}.
\endaligned
$$
From the fact that $Supp\va^j_\de\subset B_{r^m_\de}(0)$ and $$
 \displaystyle\lim_{|x|\to0}\frac{V(x)}{|x|^{\frac{1-2\tau}{\tau}}}=0,$$
we know that there is $\mathcal{E}_\de>0$ such that, for $\vr<\mathcal{E}_\de$
$$
V(\vr^{2\tau}x)\leq \frac{\vr^{2(1-2\tau)}\de}{\Lambda_\de}
$$
uniformly for $x\in B_{r^m_\de}(0)$ where
$$
\Lambda_\de:=\max\{|\va^j_\de(x)|^{2}_2:j=1,...,m\}.
$$
Similar to the arguments in Lemma \ref{AF}, we obtain
$$
\aligned
 \displaystyle \sup_{u\in H^m_{\vr\de}}I_{\vr}(u)&\leq\sum^m_{j=1}\max_{t_j\in \R}I_\vr(t_j\psi^j_\vr)\\
 &\leq\displaystyle \sum^m_{j=1}\max_{t_j\in \R}\Big\{ \frac{|t_j|^{2}}{2}\int_{\R^N}\vr^2\big|\nabla\psi^j_{\vr}\big|^2dx+\frac{|t_j|^{2}}{2}\int_{\R^N}V(x)|\psi^j_{\vr}|^2dx\\
&\hspace{1cm}-\frac{|t_j|^{2\cdot2_{\mu \ast}} }{2\cdot2_{\mu \ast}}\int_{\R^N}\int_{\R^N}\frac{|\psi^j_\vr(x)|^{2_{\mu \ast}}|\psi^j_\vr(y)|^{2_{\mu \ast}}}
{|x-y|^{\mu}}dxdy\Big\}\vspace{5mm}\\
&\leq\displaystyle m\de\max_{t\in [0,1]}\{\vr^{(2\tau-1)(N-2)}t^2-\frac{\vr^{(2\tau-1)(2N-\mu)}}{2\cdot2_{\mu \ast}}t^{2\cdot2_{\mu \ast}}\}\\
&=m\de\frac{N-\mu}{2N-\mu}2^{\frac{N}{N-\mu}}\vr^{\frac{2(1-2\tau)(2N-\mu)}{N-\mu}}
\endaligned
$$
for all $\vr\leq\mathcal{E}_\de$. Now we define the Minimax values $c_{\vr j}$ by
\begin{equation}\label{}
c_{\vr j} := \inf_{i(A)\ge j}\ \sup_{u\in A}I_\vr(u), \quad
1\le j\le m.
\end{equation}
Since $I_\vr |_{\partial B_{\rho_\vr}}\geq\kappa_\vr> 0$ and $\max I_\vr(H^m_{\vr\de})\leq m\de\frac{N-\mu}{2N-\mu}2^{\frac{N}{N-\mu}}\vr^{\frac{2(1-2\tau)(2N-\mu)}{N-\mu}}$, we know
$$
\kappa_\vr\leq c_{\vr 1}\leq\ldots\le c_{\vr m} \leq \sup_{u\in
H_{\vr m} }I_{\vr}(u)\leq m\de\frac{N-\mu}{2N-\mu}2^{\frac{N}{N-\mu}}\vr^{\frac{2(1-2\tau)(2N-\mu)}{N-\mu}}.
$$
Since $0<\tau\leq\frac{1}{N-\mu+2}$, if $\de$ is small enough, we have
$$
m\de\frac{N-\mu}{2N-\mu}2^{\frac{N}{N-\mu}}\vr^{\frac{2(1-2\tau)(2N-\mu)}{N-\mu}}\leq \frac{N-\mu+2}{2(2N-\mu)}S_{H,L}^{\frac{2N-\mu}{N-\mu+2}}\vr^{\frac{2(2N-\mu)}{N-\mu+2}}.
$$
Take $F=H^m_{\vr\de}$,  it follows from Lemma \ref{PSlemma:3} that all $c_{\vr j}$ are critical
values and $I_{\vr}$ has at least $m$ pairs of nontrivial
critical points satisfying
\[
\kappa_\vr\leq I_\vr(u_\vr)\leq m\de\frac{N-\mu}{2N-\mu}2^{\frac{N}{N-\mu}}\vr^{\frac{2(1-2\tau)(2N-\mu)}{N-\mu}}.
\]
Therefore equation \eqref{SCP2} has at least $m$ pairs of solutions. Furthermore, since $0<\tau\leq\frac{1}{N-\mu+2}$, we know
$$
\frac{2(1-2\tau)(2N-\mu)}{N-\mu}-2=\frac{2N(1-4\tau)+4\mu\tau}{N-\mu}>0,
$$
then we know
$$
\int_{\R^N}\Big(|\nabla u_\vr|^2+ V(x)|u_\vr|^2\Big)dx\leq  \de2^{\frac{2N-\mu}{N-\mu}}\vr^{\frac{2N(1-4\tau)+4\mu\tau}{N-\mu}},
$$
which means that $u_\vr$ goes to $0$, as $\vr\to0$.
\end{proof}

\section{Critical problem without lower perturbation}
In this section we assume that conditions $(V_4)$ and $(V_5)$ hold, $0<\mu<\min\{4,N\}$ and $N\geq3$. To apply the variational methods, we introduce the energy functional associated to equation \eqref{CE2} by
$$
I_{\vr}(u)=\frac{1}{2}\int_{\mathbb{R}^N}(\vr^{2}|\nabla u|^{2}+ V(x)|u|^{2})dx-\frac{1}{2\cdot2_{\mu}^{\ast}}\int_{\mathbb{R}^N}\int_{\mathbb{R}^N}\frac{|u(x)|^{2_{\mu}^{\ast}}|u(y)|^{2_{\mu}^{\ast}}}
{|x-y|^{\mu}}dxdy.
$$
The Hardy--Littlewood--Sobolev inequality implies that $I_{\vr}$ is well defined on $D^{1,2}(\mathbb{R}^N)$ and belongs to $\mathcal{C}^{1}$. And so $u$ is a weak solution of \eqref{CE2} if and only if $u$ is a critical point of the functional $I_{\vr}$.

We need to recall some basic results first. Let $\widetilde{U}_{\delta,z}(x):=\frac{[N(N-2)\delta]^{\frac{N-2}{4}}}{(\delta+|x-z|^{2})^{\frac{N-2}{2}}}$, $\delta>0$, $z\in\mathbb{R}^{N}$.
We know that $\widetilde{U}_{\delta,z}$ is a minimizer for the Sobolev best constant $S$ \cite{Wi} and
\begin{equation}\label{REL}
\aligned
U_{\delta,z}(x):=C(N,\mu)^{\frac{2-N}{2(N-\mu+2)}}S^{\frac{(N-\mu)(2-N)}{4(N-\mu+2)}}\widetilde{U}_{\delta,z}(x)
\endaligned
\end{equation}
is the unique  minimizer for $S_{H,L}$ satisfying
\begin{equation}\label{CCE1}
-\Delta u
=\left(\int_{\mathbb{R}^N}\frac{|u(y)|^{2_{\mu}^{\ast}}}{|x-y|^{\mu}}dy\right)|u|^{2_{\mu}^{\ast}-2}u\hspace{4.14mm}\mbox{in}\hspace{1.14mm} \mathbb{R}^N
\end{equation}
and
$$
\int_{\mathbb{R}^N}|\nabla U_{\delta,z}|^{2}dx=\int_{\mathbb{R}^N}\int_{\mathbb{R}^N}\frac{|U_{\delta,z}(x)|^{2_{\mu}^{\ast}}
|U_{\delta,z}(y)|^{2_{\mu}^{\ast}}}{|x-y|^{\mu}}dxdy=S_{H,L}^{\frac{2N-\mu}{N-\mu+2}}.
$$
To study the semiclassical problem, we need to consider equation
\begin{equation}\label{CCEE1}
-\vr^{2}\Delta u
=\left(\int_{\mathbb{R}^N}\frac{|u(y)|^{2_{\mu}^{\ast}}}{|x-y|^{\mu}}dy\right)|u|^{2_{\mu}^{\ast}-2}u\hspace{4.14mm}\mbox{in}\hspace{1.14mm} \mathbb{R}^N
\end{equation}
and its energy functional defined by
$$
J_{\vr}(u)=\frac{\vr^{2}}{2}\int_{\mathbb{R}^N}|\nabla u|^{2}dx-\frac{1}{2\cdot2_{\mu}^{\ast}}\int_{\mathbb{R}^N}
\int_{\mathbb{R}^N}\frac{|u(x)|^{2_{\mu}^{\ast}}|u(y)|^{2_{\mu}^{\ast}}}{|x-y|^{\mu}}dxdy.
$$
It is easy to see that $u=\vr^{\frac{N-2}{N+2-\mu}}U_{\delta,z}(x)$ is a least energy solution of \eqref{CCEE1} and
$$
J_{\vr}(u)=\vr^{\frac{2(2N-\mu)}{N+2-\mu}}\frac{N+2-\mu}{2(2N-\mu)} S_{H,L}^{\frac{2N-\mu}{N+2-\mu}}.
$$

\subsection{A nonlocal global compactness lemma}
Let
$u\rightarrow u_{r,x_{0}}=r^{\frac{N-2}{2}}u(rx+x_{0})$ be the rescaling, where $r\in\mathbb{R}^{+}$ and $x_{0}\in \mathbb{R}^{N}$. Inspired by \cite{Sm, Wi} we can establish the following global compactness lemma for nonlocal type problems.

\begin{lem}\label{F4} Suppose that conditions $(V_4)$ and $(V_5)$ hold and $N\geq3$, $0<\mu<\min\{4,N\}$. For each $\vr>0$ assume that $\{u_{n}\}\subset D^{1,2}(\mathbb{R}^N)$ is a $(PS)$ sequence for $I_{\vr}$. Then there exist a number $k\in \N$, a solution $u^{0}$ of \eqref{CE2}, solutions $u^{1},... ,u^{k}$ of \eqref{CCEE1}, sequences of points $x_{n}^{1},...,x_{n}^{k}\in\mathbb{R}^N$ and radii $r_{n}^{1},...,r_{n}^{k}>0$ such that for some subsequence $n\rightarrow\infty$
$$
\aligned
&u_{n}^{0}\equiv u_{n}\rightharpoonup u^{0}\ \ \ \mbox{weakly in}\ \ D^{1,2}(\mathbb{R}^N),\\
&u_{n}^{j}\equiv (u_{n}^{j-1}-u^{j-1})_{r_{n}^{j},x_{n}^{j}}\rightharpoonup u^{j}\ \ \ \mbox{weakly in}\ \ D^{1,2}(\mathbb{R}^N),\ \ j=1,...,k.
\endaligned
$$
Moreover as $n\rightarrow\infty$
$$
\aligned
&\|u_{n}\|^{2}\rightarrow \Sigma_{j=0}^{k}\|u^{j}\|^{2},\\
& I_{\vr}(u_{n})\rightarrow I_{\vr}(u^{0})+\displaystyle\Sigma_{j=1}^{k}J_{\vr}(u^{j}).
\endaligned
$$
\end{lem}
\begin{proof}
Since $\{u_{n}\}$ is a $(PS)$ sequence for $I_{\vr}$, we know easily that it is bounded in $D^{1,2}(\mathbb{R}^N)$. Hence we may assume that $u_{n}\rightharpoonup u^{0}$ weakly in $D^{1,2}(\mathbb{R}^N)$ as $n\rightarrow\infty$ and
that $u^{0}$ is a weak solution of \eqref{CE2}. So if we put
$$
v_{n}^{1}(x)=(u_{n}-u^{0})(x),
$$
then $v_{n}^{1}$ is a $(PS)$ sequence for $I_{\vr}$ satisfying
$$
v_{n}^{1}\rightharpoonup0\ \ \ \mbox{weakly in}\ \ D^{1,2}(\mathbb{R}^N).
$$
Then, together with the Br\'{e}zis-Lieb Lemma \cite{HBL} Lemma \ref{BLN} enable us to deduce
\begin{equation}\label{b45}
\|v_{n}^{1}\|^{2}=\|u_{n}\|^{2}-\|u^{0}\|^{2}+o(1),
\end{equation}
\begin{equation}\label{b46}
I_{\vr}(v_{n}^{1})=I_{\vr}(u_{n})-I_{\vr}(u^{0})+o(1)
\end{equation}
and
\begin{equation}\label{b47}
I_{\vr}'(v_{n}^{1})=I_{\vr}'(u_{n})-I_{\vr}'(u^{0})+o(1)=o(1).
\end{equation}
By (2.16) in \cite{BC1}, we have
\begin{equation}\label{b26}
\int_{\mathbb{R}^N}V(x)|v_{n}^{1}|^{2}dx\rightarrow0.
\end{equation}
Therefore
\begin{equation}\label{b27}
J_{\vr}(v_{n}^{1})=I_{\vr}(v_{n}^{1})+o(1)=I_{\vr}(u_{n})-I_{\vr}(u^{0})+o(1),
\end{equation}
\begin{equation}\label{b28}
J_{\vr}'(v_{n}^{1})=I_{\vr}'(v_{n}^{1})+o(1)=o(1).
\end{equation}

If $v_{n}^{1}\rightarrow0$ strongly in $D^{1,2}(\mathbb{R}^N)$ we are done: $k$ is just 0 and
$$
u_{n}^{0}:=u_{n}.
$$
Now suppose
that
$$
v_{n}^{1}\nrightarrow0\ \ \ \mbox{strongly in}\ \ D^{1,2}(\mathbb{R}^N).
$$
Moreover there exists $\zeta\in(0,\infty)$ such that
\begin{equation}\label{b29}
J_{\vr}(v_{n}^{1})\geq\zeta>0
\end{equation}
for $n$ large enough. 

 In order to complete the proof, we need to prove the following claim.
 
 {\bf Claim:} there exist sequences $\{r_{n}\}\subset\mathbb{R}^+$ and $\{y_{n}\}\subset\mathbb{R}^N$ such that
\begin{equation}\label{b4}
h_{n}=(v_{n}^{1})_{r_{n},y_{n}}\rightharpoonup h\not\equiv0 \ \ \ \mbox{weakly in}\ \ D^{1,2}(\mathbb{R}^N)
\end{equation}
as $n\rightarrow\infty$.

In fact, by \eqref{b28}, we obtain
$$
\vr^{2}\|v_{n}^{1}\|^{2}=\int_{\mathbb{R}^N}\int_{\mathbb{R}^N}
\frac{|v_{n}^{1}(x)|^{2_{\mu}^{\ast}}|v_{n}^{1}(y)|^{2_{\mu}^{\ast}}}
{|x-y|^{\mu}}dxdy+o(1).
$$
Then we can write
$$
J_{\vr}(v_{n}^{1})=\frac{N-\mu+2}{2(2N-\mu)}\int_{\mathbb{R}^N}\int_{\mathbb{R}^N}
\frac{|v_{n}^{1}(x)|^{2_{\mu}^{\ast}}|v_{n}^{1}(y)|^{2_{\mu}^{\ast}}}
{|x-y|^{\mu}}dxdy+o(1).
$$
So, by the Hardy--Littlewood--Sobolev inequality, \eqref{b29} and the boundedness of $\{u_{n}\}$, we know that $0<a_{1}<|v_{n}^{1}|_{2^{\ast}}^{2_{\mu}^{\ast}}<A_{1}$ for some $a_{1}, A_{1}>0$.
Let us define the Levy concentration function:
$$
Q_{n}(r):=\sup_{z\in\mathbb{R}^N}\int_{B_{r}(z)}|v_{n}^{1}(x)|^{2^{\ast}}dx.
$$
Since $Q_{n}(0)=0$ and $Q_{n}(\infty)>a_{1}^{\frac{2N}{2N-\mu}}$, we may assume there exists sequences $\{r_{n}\}$ and $\{y_{n}\}$ of points in $\mathbb{R}^N$ such that $r_{n}>0$ and
$$
\sup_{z\in\mathbb{R}^N}\int_{B_{r_{n}}(z)}|v_{n}^{1}(x)|^{2^{\ast}}dx
=\int_{B_{r_{n}}(y_{n})}|v_{n}^{1}(x)|^{2^{\ast}}dx=b
$$
for some
$$
0<b<\min\left\{\frac{\vr^{\frac{4N}{4-\mu}}
S^{\frac{2N}{4-\mu}}}{(2C(N,\mu)A_{1})^{\frac{2N}{4-\mu}}},a_{1}^{\frac{2N}{2N-\mu}}\right\}. $$

Let us define $h_{n}:=(v_{n}^{1})_{r_{n},y_{n}}$. We may assume that $h_{n}\rightharpoonup h$ weakly in $D^{1,2}(\mathbb{R}^N)$ and $h_{n}\rightarrow h$ a.e. on $\mathbb{R}^N$. It is easy to see that
$$
\sup_{z\in\mathbb{R}^N}\int_{B_{1}(z)}|h_{n}(x)|^{2^{\ast}}dx
=\int_{B_{1}(0)}|h_{n}(x)|^{2^{\ast}}dx=b.
$$
By invariance of the $D^{1,2}(\mathbb{R}^N)$ norms under translation and dilation, we get
$$
\|v_{n}^{1}\|=\|h_{n}\|, \ \ |v_{n}^{1}|_{2^{\ast}}=|h_{n}|_{2^{\ast}}
$$
and
$$
\int_{\mathbb{R}^N}\int_{\mathbb{R}^N}\frac{|v_{n}^{1}(x)|^{2_{\mu}^{\ast}}|v_{n}^{1}(y)|^{2_{\mu}^{\ast}}}
{|x-y|^{\mu}}dxdy=\int_{\mathbb{R}^N}\int_{\mathbb{R}^N}\frac{|h_{n}(x)|^{2_{\mu}^{\ast}}
|h_{n}(y)|^{2_{\mu}^{\ast}}}
{|x-y|^{\mu}}dxdy.
$$
By direct calculation, we have
\begin{equation}\label{b31}
J_{\vr}(h_{n})=J_{\vr}(v_{n}^{1})=I_{\vr}(u_{n})+o(1)
\end{equation}
and
\begin{equation}\label{b32}
J_{\vr}'(h_{n})=J_{\vr}'(v_{n}^{1})=o(1).
\end{equation}
If $h=0$ then $h_{n}\rightarrow0$ strongly in $L_{loc}^{2}(\mathbb{R}^N)$. Let $\psi\in \mathcal{C}_{0}^{\infty}(\mathbb{R}^N)$ be such that $Supp \psi\subset B_{1}(y)$ for some $y\in\mathbb{R}^N$. Then, we have
$$
\aligned
\vr^{2}\int_{\mathbb{R}^N}|\nabla(\psi h_{n})|^{2}dx&=\vr^{2}\int_{\mathbb{R}^N}\nabla h_{n}\nabla(\psi^{2} h_{n})dx+o(1)\\
&=\int_{\mathbb{R}^N}\int_{\mathbb{R}^N}
\frac{|h_{n}(x)|^{2_{\mu}^{\ast}}|\psi(y)|^{2}|h_{n}(y)|^{2_{\mu}^{\ast}}}
{|x-y|^{\mu}}dxdy+o(1)\\
&\leq C(N,\mu)
|h_{n}|_{2^{\ast}}^{2_{\mu}^{\ast}}\Big(\int_{\mathbb{R}^N}
(|\psi|^{2}|h_{n}|^{2_{\mu}^{\ast}})^{\frac{2N}{2N-\mu}}dx\Big)^{\frac{2N-\mu}{2N}}+o(1)\\
&= C(N,\mu)
|h_{n}|_{2^{\ast}}^{2_{\mu}^{\ast}}\Big(\int_{\mathbb{R}^N}
|\psi h_{n}|^{\frac{4N}{2N-\mu}}
|h_{n}|^{\frac{2N(4-\mu)}{(2N-\mu)(N-2)}}dx\Big)^{\frac{2N-\mu}{2N}}+o(1)\\
&\leq C(N,\mu)|h_{n}|_{2^{\ast}}^{2_{\mu}^{\ast}}
|h_{n}|_{L^{2^{\ast}}(B_{1}(y))}^{2_{\mu}^{\ast}-2}
\frac{1}{S}\int_{\mathbb{R}^N}|\nabla(\psi h_{n})|^{2}dx+o(1)\\
&\leq C(N,\mu)b^{\frac{2_{\mu}^{\ast}-2}{2^{\ast}}}
\frac{A_{1}}{S}\int_{\mathbb{R}^N}|\nabla(\psi h_{n})|^{2}dx+o(1)\\
&\leq\frac{\vr^{2}}{2}\int_{\mathbb{R}^N}|\nabla(\psi h_{n})|^{2}dx+o(1)
\endaligned
$$
thanks to $0<\mu<\min\{4,N\}$. We obtain $\nabla h_{n}\rightarrow0$ strongly in $L_{loc}^{2}(\mathbb{R}^N)$ and $h_{n}\rightarrow0$ strongly in $L_{loc}^{2^{\ast}}(\mathbb{R}^N)$, which contradicts with $\displaystyle\int_{B_{1}(0)}|h_{n}(x)|^{2^{\ast}}dx=b>0$. So, $h\neq0$. By \eqref{b28} and weakly sequentially continuous $J_{\vr}'$, we know $h$ solves \eqref{CCEE1} weakly. The sequences $\{h_{n}\}$, $\{r_{n}^{1}\}$, and $\{y_{n}^{1}\}$ are the wanted sequences.

By iteration, we obtain sequences $v_{n}^{j}=u_{n}^{j-1}-u^{j-1}$, $j\geq2$, and the rescaled functions $u_{n}^{j}=(v_{n}^{j})_{r_{n}^{j},y_{n}^{j}}\rightharpoonup u^{j}$ weakly in $ D^{1,2}(\mathbb{R}^N)$, where
each $u^{j}$ solves \eqref{CCEE1}. Moreover, from \eqref{b45}, \eqref{b27} and \eqref{b28}, we know by induction  that
\begin{equation}\label{b5}
\|u_{n}^{j}\|^{2}=\|v_{n}^{j}\|^{2}=\|u_{n}^{j-1}\|^{2}-\|u^{j-1}\|^{2}+o(1)=...=
\|u_{n}\|^{2}-\Sigma_{i=0}^{j-1}\|u^{i}\|^{2}+o(1)
\end{equation}
and
\begin{equation}\label{b6}
J_{\vr}(u_{n}^{j})=J_{\vr}(v_{n}^{j})+o(1)=J_{\vr}(u_{n}^{j-1})
-J_{\vr}(u^{j-1})+o(1)=...
=I_{\vr}(u_{n})-I_{\vr}(u^{0})-\Sigma_{i=1}^{j-1}J_{\vr}(u^{i})+o(1).
\end{equation}
Furthermore, from the estimate
$$
0=\langle J_{\vr}'(u^{j}),u^{j}\rangle=\vr^{2}\|u^{j}\|^{2}-\int_{\mathbb{R}^N}
\int_{\mathbb{R}^N}\frac{|u^{j}(x)|^{2_{\mu}^{\ast}}|u^{j}(y)|^{2_{\mu}^{\ast}}}{|x-y|^{\mu}}dxdy
\geq\|u^{j}\|^{2}(\vr^{2}-S_{H,L}^{-2_{\mu}^{\ast}}\|u^{j}\|^{2\cdot2_{\mu}^{\ast}-2}),
$$
we see that $\|u^{j}\|^{2}\geq \vr^{\frac{2N-4}{N-\mu+2}}S_{H,L}^{\frac{2N-\mu}{N-\mu+2}}$ and the iteration must terminate at some index $k\geq0$ due to \eqref{b5}.
\end{proof}

\begin{Prop}\label{NE}
Suppose that conditions $(V_4)$ and $(V_5)$ hold. For each $\vr>0$, define $$
\mathcal{N}_{\vr}=\Big\{u\in D^{1,2}(\mathbb{R}^N)\backslash\{0\}:\langle I_{\vr}'(u),u\rangle=0\Big\},
$$
then the minimization problem
\begin{equation}\label{a3}
\inf\Big\{I_{\vr}(u):u\in\mathcal{N}_{\vr}\Big\}
\end{equation}
has no solution.
\end{Prop}
\begin{proof}
Let us denote by $S_{\mathcal{N}_{\vr}}$ the infimum defined by \eqref{a3}. Obviously, we have for any $u\in\mathcal{N}_{\vr}$
$$\aligned
\vr^{2}S_{H,L}&\left(\int_{\mathbb{R}^N}\int_{\mathbb{R}^N}
\frac{|u(x)|^{2_{\mu}^{\ast}}|u(y)|^{2_{\mu}^{\ast}}}
{|x-y|^{\mu}}dxdy\right)^{\frac{N-2}{2N-\mu}}\\
&\leq\int_{\mathbb{R}^N}\vr^{2}|\nabla u|^{2}dx\\
&\leq\int_{\mathbb{R}^N}(\vr^{2}|\nabla u|^{2}+ V(x)|u|^{2})dx\\
&=\int_{\mathbb{R}^N}\int_{\mathbb{R}^N}
\frac{|u(x)|^{2_{\mu}^{\ast}}|u(y)|^{2_{\mu}^{\ast}}}
{|x-y|^{\mu}}dxdy
\endaligned
$$
which means
$$
\int_{\mathbb{R}^N}\int_{\mathbb{R}^N}
\frac{|u(x)|^{2_{\mu}^{\ast}}|u(y)|^{2_{\mu}^{\ast}}}
{|x-y|^{\mu}}dxdy\geq\vr^{\frac{2(2N-\mu)}{N+2-\mu}} S_{H,L}^{\frac{2N-\mu}{N+2-\mu}}.
$$
Thus, for any $u\in\mathcal{N}_{\vr}$,
$$
I_{\vr}(u)=\frac{N+2-\mu}{2(2N-\mu)}\int_{\mathbb{R}^N}\int_{\mathbb{R}^N}
\frac{|u(x)|^{2_{\mu}^{\ast}}|u(y)|^{2_{\mu}^{\ast}}}
{|x-y|^{\mu}}dxdy\geq \vr^{\frac{2(2N-\mu)}{N+2-\mu}}\frac{N+2-\mu}{2(2N-\mu)} S_{H,L}^{\frac{2N-\mu}{N+2-\mu}}
$$
and we get
$$
S_{\mathcal{N}_{\vr}}\geq \vr^{\frac{2(2N-\mu)}{N+2-\mu}}\frac{N+2-\mu}{2(2N-\mu)} S_{H,L}^{\frac{2N-\mu}{N+2-\mu}}.
$$

We shall show that the equality holds indeed. Let us consider
the sequence
$$
u_{n}=\vr^{\frac{N-2}{N+2-\mu}}U_{\frac{1}{n},0}(x).
$$
Then,
$$
v_{n}:=\Big(1+\vr^{\frac{2\mu-4N}{N+2-\mu}} S_{H,L}^{\frac{\mu-2N}{N+2-\mu}}
\int_{\mathbb{R}^{N}}V(x)|u_{n}|^{2}dx\Big)^{\frac{N-2}{2(N+2-\mu)}}u_{n}\in\mathcal{N}_{\vr}
$$
with
$$
I_{\vr}(v_{n})=\vr^{\frac{2(2N-\mu)}{N+2-\mu}}\frac{N+2-\mu}{2(2N-\mu)} S_{H,L}^{\frac{2N-\mu}{N+2-\mu}}\Big(1+\vr^{\frac{2\mu-4N}{N+2-\mu}} S_{H,L}^{\frac{\mu-2N}{N+2-\mu}}
\int_{\mathbb{R}^{N}}V(x)|u_{n}|^{2}dx\Big)^{\frac{2N-\mu}{N+2-\mu}}.
$$
By Lemma 3.2 in \cite{BC1}, we know
$$
\lim\limits_{n\rightarrow\infty}\displaystyle\int_{\mathbb{R}^N}V(x)
U_{\frac{1}{n},0}^{2}dx=0.
$$
Thus,
$$
I_{\vr}(v_{n})\rightarrow\vr^{\frac{2(2N-\mu)}{N+2-\mu}}\frac{N+2-\mu}{2(2N-\mu)} S_{H,L}^{\frac{2N-\mu}{N+2-\mu}},
$$
as $n\rightarrow\infty$ and
$$
S_{\mathcal{N}_{\vr}}\leq \vr^{\frac{2(2N-\mu)}{N+2-\mu}}\frac{N+2-\mu}{2(2N-\mu)} S_{H,L}^{\frac{2N-\mu}{N+2-\mu}}.
$$
So, we can obtain
$$
S_{\mathcal{N}_{\vr}}= \vr^{\frac{2(2N-\mu)}{N+2-\mu}}\frac{N+2-\mu}{2(2N-\mu)} S_{H,L}^{\frac{2N-\mu}{N+2-\mu}}.
$$

Now we can argue by contradiction to prove the nonexistence result.
Let $u\in\mathcal{N}_{\vr}$ be a function such that
$$
I_{\vr}(u)=\vr^{\frac{2(2N-\mu)}{N+2-\mu}}\frac{N+2-\mu}{2(2N-\mu)} S_{H,L}^{\frac{2N-\mu}{N+2-\mu}},
$$
then
$$
\int_{\mathbb{R}^N}\int_{\mathbb{R}^N}
\frac{|u(x)|^{2_{\mu}^{\ast}}|u(y)|^{2_{\mu}^{\ast}}}
{|x-y|^{\mu}}dxdy=\vr^{\frac{2(2N-\mu)}{N+2-\mu}} S_{H,L}^{\frac{2N-\mu}{N+2-\mu}}.
$$
By the definition of $S_{H,L}$, we know
$$
\int_{\mathbb{R}^N}\vr^{2}|\nabla u|^{2}dx\geq\vr^{\frac{2(2N-\mu)}{N+2-\mu}} S_{H,L}^{\frac{2N-\mu}{N+2-\mu}}.
$$
Thus $\displaystyle\int_{\mathbb{R}^{N}}V(x)|u|^{2}dx=0$ and  we have $u\equiv0$ on $\mathbb{R}^{N}\backslash M$ by condition $(V_{4})$. Consequently
$$
\int_{M}\int_{M}
\frac{|u(x)|^{2_{\mu}^{\ast}}|u(y)|^{2_{\mu}^{\ast}}}
{|x-y|^{\mu}}dxdy=
\int_{M}\vr^{2}|\nabla u|^{2}dx=\vr^{\frac{2(2N-\mu)}{N+2-\mu}} S_{H,L}^{\frac{2N-\mu}{N+2-\mu}}
$$
which means
$$
\frac{\displaystyle\int_{M}|\nabla u|^{2}dx}{\Big(\displaystyle\int_{M}\int_{M}
\frac{|u(x)|^{2_{\mu}^{\ast}}|u(y)|^{2_{\mu}^{\ast}}}
{|x-y|^{\mu}}dxdy\Big)^{\frac{N-2}{2N-\mu}}}=S_{H,L}.
$$
However, from Lemma 1.3 in \cite{GY} we know $S_{H,L}(\Omega)=S_{H,L}$ is never achieved except when $\Omega=\R^N$, where
$$
S_{H,L}(\Omega):=\displaystyle\inf\limits_{u\in D_{0}^{1,2}(\Omega)\backslash\{{0}\}}\ \ \frac{\displaystyle\int_{\Omega}|\nabla u|^{2}dx}{\left(\displaystyle\int_{\Omega}\int_{\Omega}\frac{|u(x)|^{2_{\mu}^{\ast}}
|u(y)|^{2_{\mu}^{\ast}}}{|x-y|^{\mu}}dxdy\right)^{\frac{N-2}{2N-\mu}}}.
$$
So in conclusion, we know that $S_{\mathcal{N}_{\vr}}$ is not attained.
\end{proof}

\begin{cor}\label{PS2}
Let $\{u_{n}\}\subset D^{1,2}(\mathbb{R}^N)$ be a $(PS)_{c}$ sequence for $I_{\vr}$ with
$$\vr^{\frac{2(2N-\mu)}{N+2-\mu}}\frac{N+2-\mu}{2(2N-\mu)} S_{H,L}^{\frac{2N-\mu}{N+2-\mu}}<c<
\vr^{\frac{2(2N-\mu)}{N+2-\mu}}\frac{N+2-\mu}{2N-\mu}S_{H,L}^{\frac{2N-\mu}{N+2-\mu}}.
$$
Then for each $\vr>0$ $\{u_{n}\}$ is relatively compact in $D^{1,2}(\mathbb{R}^N)$.
\end{cor}
\begin{proof}
We know from Lemma \ref{F4} that there exist a number $k\in \N$, a solution $u^{0}$ of \eqref{CE2} and solutions $u^{1},... ,u^{k}$ of \eqref{CCEE1}, such that for some subsequence $n\rightarrow\infty$
$$
\aligned
&\|u_{n}\|^{2}\rightarrow \Sigma_{j=0}^{k}\|u^{j}\|^{2},\\
&I_{\vr}(u_{n})\rightarrow I_{\vr}(u^{0})+\Sigma_{j=1}^{k}J_{\vr}(u^{j}).
\endaligned
$$
By Proposition \ref{NE}, if $u$ is a nontrivial solution of \eqref{CE2}, then
$$
I_{\vr}(u)>\vr^{\frac{2(2N-\mu)}{N+2-\mu}}\frac{N+2-\mu}{2(2N-\mu)} S_{H,L}^{\frac{2N-\mu}{N+2-\mu}}.
$$
While for every nontrivial solution $v$ of \eqref{CCEE1}
$$
J_{\vr}(v)\geq\vr^{\frac{2(2N-\mu)}{N+2-\mu}}\frac{N+2-\mu}{2(2N-\mu)} S_{H,L}^{\frac{2N-\mu}{N+2-\mu}}.
$$
Since
$$c<
\vr^{\frac{2(2N-\mu)}{N+2-\mu}}\frac{N+2-\mu}{2N-\mu}S_{H,L}^{\frac{2N-\mu}{N+2-\mu}},
$$ we have $k=0$ or $k=1$ with $u^{0}=0$.
In conclusion, $\{u_{n}\}$ is relatively compact in $D^{1,2}(\mathbb{R}^N)$.
\end{proof}

\subsection{High energy semiclassical states}

We recall that
$$
M_{\tau}=\{x\in \mathbb{R}^N:dist(x, M)\leq\tau\}
$$
for $\tau>0$ small, we may choose $\rho>0$ such that $M_{\tau}\subset B_{\frac{\rho}{2}}(0)$, $\rho=\rho(\tau)$. Let
$$
\chi(x)=\left\{\begin{array}{l}
\displaystyle x\hspace{5.84mm}\mbox{for}\hspace{1.14mm} |x|\leq\rho,\\
\displaystyle \frac{\rho x}{|x|}\hspace{3.14mm}\mbox{for}\hspace{1.14mm} |x|>\rho.
\end{array}
\right.
$$
We define a "barycenter" $\beta(u):\mathcal{N}_{\vr}\rightarrow \mathbb{R}^N$ by
$$
\beta(u):=S_{H,L}^{\frac{\mu-2N}{N+2-\mu}}\vr^{\frac{4-2N}{N+2-\mu}}\int_{\mathbb{R}^N}\chi(x)|\nabla u|^{2}dx
$$
and a mapping $\Phi_{\delta,z}:\mathbb{R}^N\rightarrow\mathcal{N}_{\vr}$ by $$\Phi_{\delta,z}(x):=\Big(1+\vr^{\frac{2\mu-4N}{N+2-\mu}} S_{H,L}^{\frac{\mu-2N}{N+2-\mu}}
\int_{\mathbb{R}^{N}}V(x)|U_{\delta,z}|^{2}dx\Big)^{\frac{N-2}{2(N+2-\mu)}}
\vr^{\frac{N-2}{N+2-\mu}}U_{\delta,z}(x),$$
where ${U}_{\delta,z}(x)$ is defined in \eqref{REL}.

We also introduce the set
$$
\Gamma=\Gamma(\rho, \delta_{1},\delta_{2})=\{(x,\delta)\in\mathbb{R}^N\times\mathbb{R}:
|x|<\frac{\rho}{2},\delta_{1}<\delta<\delta_{2}\}.
$$
By Lemma 3.2 in \cite{BC1}, we know for any fixed $z\in\mathbb{R}^N$ there holds
$$
\lim\limits_{\delta\rightarrow0}\displaystyle\int_{\mathbb{R}^N}V(x)
{U}_{\delta,z}(x)^{2}dx=0,
$$
thus for every $\vr> 0$ there exist $\delta_{1}=\delta_{1}(\vr)$ and $\delta_{2}=\delta_{2}(\vr)$ with $\delta_{1}<\delta_{2}$ and $\delta_{1},\delta_{2}\to0$ as $\vr\to0$, such that
\begin{equation}\label{e2}
\sup\left\{I_{\vr}(\Phi_{\delta,z}):(z,\delta)\in\Gamma\right\}
<\vr^{\frac{2(2N-\mu)}{N+2-\mu}}\left[\frac{N+2-\mu}{2(2N-\mu)} S_{H,L}^{\frac{2N-\mu}{N+2-\mu}}+h(\vr)\right],
\end{equation}
where $h(\vr)\to0$ as $\vr\rightarrow0$.

\begin{lem}\label{E3} Set
$$
\gamma(u):=S_{H,L}^{\frac{\mu-2N}{N+2-\mu}}\vr^{\frac{4-2N}{N+2-\mu}}
\int_{\mathbb{R}^N}|\chi(x)-\beta(u)||\nabla u|^{2}dx,
$$
we have $\lim_{\delta\rightarrow0}\gamma(\Phi_{\delta,z})=0$ uniformly for $|z|\leq\frac{\rho}{2}$.
\end{lem}
\begin{proof}Note that the functional $\gamma$ measures the concentration of a function $u$ near its barycenter.
To study the behavior $\gamma(\Phi_{\delta,z})$ as $\delta\rightarrow0$. We rewrite
$$\aligned
\gamma(\Phi_{\delta,z}(x))&=S_{H,L}^{\frac{\mu-2N}{N+2-\mu}}\vr^{\frac{4-2N}{N+2-\mu}}
\int_{\mathbb{R}^N}|\chi(x)
-\beta(\Phi_{\delta,z})|
|\nabla \Phi_{\delta,z}(x)|^{2}dx\\
&=S_{H,L}^{\frac{\mu-2N}{N+2-\mu}}\vr^{\frac{4-2N}{N+2-\mu}}
\int_{B_{\xi}(z)}|\chi(x)-\beta(\Phi_{\delta,z})|
|\nabla \Phi_{\delta,z}(x)|^{2}dx\\
&\ \ \ \ +S_{H,L}^{\frac{\mu-2N}{N+2-\mu}}\vr^{\frac{4-2N}{N+2-\mu}}
\int_{\mathbb{R}^N\backslash B_{\xi}(z)}
|\chi(x)-\beta(\Phi_{\delta,z})|
|\nabla \Phi_{\delta,z}(x)|^{2}dx,
\endaligned$$
where $0<2\xi<\rho$.
On one hand, since
\begin{equation}\label{e3}
\int_{\mathbb{R}^N\backslash B_{\xi}(0)}|\nabla \Phi_{\delta,0}|^{2}dx=C_{\vr}
\int_{\mathbb{R}^N\backslash B_{\xi}(0)}\frac{\delta^{\frac{N-2}{2}}|x|^{2}}{(\delta+|x|^{2})^{N}}dx
=C_{\vr}
\int_{|x|\geq\frac{\xi}{\sqrt{\delta}}}\frac{|x|^{2}}{(1+|x|^{2})^{N}}dx\rightarrow0
\end{equation}
as $\delta\rightarrow0$, we know
$$
\lim_{\delta\rightarrow0}\int_{\mathbb{R}^N\backslash B_{\xi}(z)}
|\chi(x)-\beta(\Phi_{\delta,z})|
|\nabla \Phi_{\delta,z}(x)|^{2}dx=0.
$$
On the other hand, since
\begin{equation}\label{e1}
\aligned
\beta(\Phi_{\delta,z})&=S_{H,L}^{\frac{\mu-2N}{N+2-\mu}}\vr^{\frac{4-2N}{N+2-\mu}}\int_{\mathbb{R}^N}\chi(x)|\nabla \Phi_{\delta,z}(x)|^{2}dx\\
&=\Big(1+\vr^{\frac{2\mu-4N}{N+2-\mu}} S_{H,L}^{\frac{\mu-2N}{N+2-\mu}}
\int_{\mathbb{R}^{N}}V(x)|U_{\delta,z}|^{2}dx\Big)^{\frac{N-2}{N+2-\mu}}\left[z+
S_{H,L}^{\frac{\mu-2N}{N+2-\mu}}
\int_{\mathbb{R}^N}(\chi(\sqrt{\delta} y+z)-z)|\nabla {U}_{1,0}(y)|^{2}dy\right]\\
&=\Big(1+o(1)\Big)^{\frac{N-2}{N+2-\mu}}z+o(1)
\endaligned
\end{equation}
as $\delta\rightarrow0$, we know
$$\aligned
\int_{B_{\xi}(z)}&|\chi(x)-\beta(\Phi_{\delta,z})|
|\nabla \Phi_{\delta,z}(x)|^{2}dx\\
&\leq\int_{B_{\xi}(z)}|\chi(x)-\chi(z)|
|\nabla \Phi_{\delta,z}(x)|^{2}dx+\int_{B_{\xi}(z)}
|\chi(z)-\beta(\Phi_{\delta,z})|
|\nabla \Phi_{\delta,z}(x)|^{2}dx\\
&\leq2\int_{B_{\xi}(z)}|x-z|
|\nabla \Phi_{\delta,z}(x)|^{2}dx
+\Big(1+\vr^{\frac{2\mu-4N}{N+2-\mu}} S_{H,L}^{\frac{\mu-2N}{N+2-\mu}}
\int_{\mathbb{R}^{N}}V(x)|U_{\delta,z}|^{2}dx\Big)^{\frac{N-2}{N+2-\mu}}
\vr^{\frac{2N-4}{N+2-\mu}}2\xi S_{H,L}^{\frac{2N-\mu}{N+2-\mu}}+o(1),
\endaligned$$
where we had used Lemma 2 of \cite{CY} which says
$$
|\chi(x)-\chi(z)|\leq2|x-z|+2\xi,\ \ x\in B_{\xi}(z).
$$
Since $\xi> 0$ is arbitrary, $\lim_{\delta\rightarrow0}\gamma(\Phi_{\delta,z})=0$.  Finally the conclusion follows from the compactness of $\{z:|z|\leq\frac{\rho}{2}\}$.
\end{proof}

We now define a set $\tilde{\mathcal{N}}_{\vr}\subset\mathcal{N}_{\vr}$ by
$$
\tilde{\mathcal{N}}_{\vr}=\left\{u\in\mathcal{N}_{\vr}:\vr^{\frac{2(2N-\mu)}{N+2-\mu}}
\frac{N+2-\mu}{2(2N-\mu)} S_{H,L}^{\frac{2N-\mu}{N+2-\mu}}<I_{\vr}(u)
<\vr^{\frac{2(2N-\mu)}{N+2-\mu}}\left[\frac{N+2-\mu}{2(2N-\mu)} S_{H,L}^{\frac{2N-\mu}{N+2-\mu}}+h(\vr)\right],
(\beta(u),\gamma(u))\in\Gamma\right\},
$$
where $\Gamma$ has been chosen to meet \eqref{e2}. According to Lemma \ref{E3} we can select
$\delta_{1}(\vr)$ and $\delta_{2}(\vr)$ such that $\tilde{\mathcal{N}}_{\vr}\neq\emptyset$ for $\vr> 0$ small.

\begin{lem}\label{E4} We have
\begin{equation}\label{e4}
\lim\limits_{\vr\rightarrow0}\sup_{u\in \tilde{\mathcal{N}}_{\vr}}\inf_{z\in M_{\tau}}[\beta(u)-\beta(\Phi_{\delta,z})]=0.
\end{equation}
\end{lem}
\begin{proof}
Let $\vr_n\rightarrow0$, for
every $n$ there exists $u_{n}\in\tilde{\mathcal{N}}_{\vr_n}$ such that
$$
\inf_{z\in M_{\tau}}[\beta(u_{n})-\beta(\Phi_{\vr_n,z})]=
\sup_{u\in \tilde{\mathcal{N}}_{\vr}}\inf_{z\in M_{\tau}}[\beta(u)-\beta(\Phi_{\vr_n,z})]+o(1).
$$
In order to prove \eqref{e4} it is sufficient to find a sequence $\{z_{n}\}\subset M_{\tau}$ such that
\begin{equation}\label{e5}
\lim\limits_{n\rightarrow\infty}[\beta(u_{n})-\beta(\Phi_{\vr_n,z_{n}})]=0.
\end{equation}

Since for any $u\in\mathcal{N}_\vr$
$$
\int_{\mathbb{R}^N}\int_{\mathbb{R}^N}
\frac{|u(x)|^{2_{\mu}^{\ast}}|u(y)|^{2_{\mu}^{\ast}}}
{|x-y|^{\mu}}dxdy\geq\vr^{\frac{2(2N-\mu)}{N+2-\mu}} S_{H,L}^{\frac{2N-\mu}{N+2-\mu}},
$$
then we have
$$
\int_{\mathbb{R}^N}\vr^{2}|\nabla u|^{2}dx\geq\vr^{2}S_{H,L}\left(\int_{\mathbb{R}^N}\int_{\mathbb{R}^N}
\frac{|u(x)|^{2_{\mu}^{\ast}}|u(y)|^{2_{\mu}^{\ast}}}
{|x-y|^{\mu}}dxdy\right)^{\frac{N-2}{2N-\mu}}
\geq\vr^{\frac{2(2N-\mu)}{N+2-\mu}} S_{H,L}^{\frac{2N-\mu}{N+2-\mu}}.
$$
And so, for $u_{n}\in\tilde{\mathcal{N}}_{\vr_n}$, we know
$$\aligned
\vr_n^{\frac{2(2N-\mu)}{N+2-\mu}} \frac{N+2-\mu}{2(2N-\mu)} S_{H,L}^{\frac{2N-\mu}{N+2-\mu}}&\leq\frac{N+2-\mu}{2(2N-\mu)}
\vr_n^{2}\int_{\mathbb{R}^N}|\nabla u_{n}|^{2}dx\\
&\leq \frac{N+2-\mu}{2(2N-\mu)}\int_{\mathbb{R}^N}(\vr_n^{2}|\nabla u_{n}|^{2}+ V(x)|u_{n}|^{2})dx\\
&=I_{\vr_n}(u_{n})\\
&<\vr_n^{\frac{2(2N-\mu)}{N+2-\mu}}\left[\frac{N+2-\mu}{2(2N-\mu)} S_{H,L}^{\frac{2N-\mu}{N+2-\mu}}+h(\vr_n)\right].
\endaligned$$
Let $g_{n}:=\vr_n^{\frac{2-N}{N-\mu+2}}u_{n}$, we obtain
$$
 S_{H,L}^{\frac{2N-\mu}{N+2-\mu}}\leq \int_{\mathbb{R}^N}|\nabla g_{n}|^{2}dx
\leq \int_{\mathbb{R}^N}|\nabla g_{n}|^{2}dx+ \frac{1}{\vr_n^{2}}\int_{\mathbb{R}^N}V(x)|g_{n}|^{2}dx
< S_{H,L}^{\frac{2N-\mu}{N+2-\mu}}+\frac{2(2N-\mu)}{N+2-\mu}h(\vr_n).
$$
Hence
$$
\lim\limits_{n\rightarrow\infty}\int_{\mathbb{R}^N}|\nabla g_{n}|^{2}dx=S_{H,L}^{\frac{2N-\mu}{N+2-\mu}}, \ \
\lim\limits_{n\rightarrow\infty}\int_{\mathbb{R}^N} V(x)|g_{n}|^{2}dx=0
$$
and
$$
\lim\limits_{n\rightarrow\infty}\int_{\mathbb{R}^N}\int_{\mathbb{R}^N}
\frac{|g_{n}(x)|^{2_{\mu}^{\ast}}|g_{n}(y)|^{2_{\mu}^{\ast}}}
{|x-y|^{\mu}}dxdy=S_{H,L}^{\frac{2N-\mu}{N+2-\mu}}.
$$
Then $\{g_{n}\}$ is a $(PS)$ sequence for $I_{1}$. It then follows from Lemma \ref{F4} with $\vr=1$ and Proposition \ref{NE} that there exist a number $k\in \N$ and solutions $g^{1},... ,g^{k}$ of \eqref{CCEE1}, sequences of points $x_{n}^{1},...,x_{n}^{k}\in\mathbb{R}^N$ and radii $r_{n}^{1},...,r_{n}^{k}>0$ such that for some subsequence $n\rightarrow\infty$
$$
\aligned
&g_{n}^{0}\equiv g_{n}\rightharpoonup g^{0}=0\ \ \ \mbox{weakly in}\ \ D^{1,2}(\mathbb{R}^N),\\
&g_{n}^{j}\equiv (g_{n}^{j-1}-g^{j-1})_{r_{n}^{j},x_{n}^{j}}\rightharpoonup g^{j}\ \ \ \mbox{weakly in}\ \ D^{1,2}(\mathbb{R}^N),\ \ j=1,...,k,
\endaligned
$$
and
$$
\aligned
&\|g_{n}\|^{2}\rightarrow \Sigma_{j=1}^{k}\|g^{j}\|^{2},\\
&I_{1}(g_{n})\rightarrow \Sigma_{j=1}^{k}J_{1}(g^{j}).
\endaligned
$$
Since
$$
I_{1}(g_{n})\rightarrow\frac{N+2-\mu}{2(2N-\mu)} S_{H,L}^{\frac{2N-\mu}{N+2-\mu}} \ \ \mbox{and} \ \
J_{1}(g^{j})\geq \frac{N+2-\mu}{2(2N-\mu)}S_{H,L}^{\frac{2N-\mu}{N+2-\mu}},
$$
we must have $k=1$ with $\|g_{n}\|^{2}\rightarrow \|g^{1}\|^{2}$ and $J_{1}(g^{1})= \frac{N+2-\mu}{2(2N-\mu)} S_{H,L}^{\frac{2N-\mu}{N+2-\mu}}$. Recall that any solution of \eqref{CCE1} must be of the form
$$
U_{\delta,z}(x)=C(N,\mu)^{\frac{2-N}{2(N-\mu+2)}}S^{\frac{(N-\mu)(2-N)}{4(N-\mu+2)}}
\frac{[N(N-2)\delta]^{\frac{N-2}{4}}}{(\delta+|x-z|^{2})^{\frac{N-2}{2}}}, \ \delta>0, \ z\in\mathbb{R}^{N},
$$
then we know
$$
g^{1}=C(N,\mu)^{\frac{2-N}{2(N-\mu+2)}}S^{\frac{(N-\mu)(2-N)}{4(N-\mu+2)}}
\frac{[N(N-2)\delta^{1}]^{\frac{N-2}{4}}}{(\delta^{1}+|x-z^{1}|^{2})^{\frac{N-2}{2}}}
$$
for some $\delta^{1}>0$ and $z^{1}\in\mathbb{R}^{N}$ and there exist a sequence of points
$\{z_{n}\}\subset \mathbb{R}^N$ and a sequence $\{\tau_{n}\}\subset(0,\infty)$ such that
$$
\|U_{\tau_{n},z_{n}}-g_{n}\|\rightarrow0
$$
i.e.
$$
\|U_{\tau_{n},z_{n}}-\vr_n^{\frac{2-N}{N-\mu+2}}u_{n}\|\rightarrow0.
$$
And so
$$
\|\vr_n^{\frac{N-2}{N-\mu+2}}U_{\tau_{n},z_{n}}-u_{n}\|\rightarrow0.
$$
Denote $\Psi_{\tau_{n},z_{n}}:=\vr_n^{\frac{N-2}{N-\mu+2}}U_{\tau_{n},z_{n}}$ and $w_{n}:=u_{n}-\Psi_{\tau_{n},z_{n}}$ and then $\|w_{n}\|\rightarrow0$ and
$$
u_{n}(x)=w_{n}(x)+\Psi_{\tau_{n},z_{n}}(x) \ \ \mbox{on} \ \ \mathbb{R}^N.
$$

We claim that the sequence $\tau_{n}\rightarrow0$ and $\{z_{n}\}$ is bounded. Then we may suppose that $z_{n}\rightarrow \overline{z}$. Since $\tau_{n}\rightarrow0$, by the definition of $ \Phi_{\tau_{n},z_{n}}$ and $\Psi_{\tau_{n},z_{n}}$ we have $\|\Phi_{\tau_{n},z_{n}}-\Psi_{\tau_{n},z_{n}}\|\rightarrow0$. Denote $f_{n}:=u_{n}-\Phi_{\tau_{n},z_{n}}$, then $\|f_{n}\|\rightarrow0$. Consequently, we have
$$\aligned
\beta(u_{n})&=S_{H,L}^{\frac{\mu-2N}{N+2-\mu}}\vr_n^{\frac{4-2N}{N+2-\mu}}
\int_{\mathbb{R}^N}\chi(x)|\nabla u_{n}|^{2}dx\\
&=S_{H,L}^{\frac{\mu-2N}{N+2-\mu}}\vr_n^{\frac{4-2N}{N+2-\mu}}
\int_{\mathbb{R}^N}\chi(x)|\nabla (f_{n}+\Phi_{\tau_{n},z_{n}})|^{2}dx\\
&=S_{H,L}^{\frac{\mu-2N}{N+2-\mu}}\vr_n^{\frac{4-2N}{N+2-\mu}}
\int_{\mathbb{R}^N}\chi(x)|\nabla \Phi_{\tau_{n},z_{n}}|^{2}dx\\
&\hspace{0.5cm}+S_{H,L}^{\frac{\mu-2N}{N+2-\mu}}\vr_n^{\frac{4-2N}{N+2-\mu}}
\int_{\mathbb{R}^N}\chi(x)|\nabla f_{n}|^{2}dx\\
&\hspace{1cm}+2S_{H,L}^{\frac{\mu-2N}{N+2-\mu}}\vr_n^{\frac{4-2N}{N+2-\mu}}
\int_{\mathbb{R}^N}\chi(x)\nabla f_{n}\nabla\Phi_{\tau_{n},z_{n}}dx\\
&=S_{H,L}^{\frac{\mu-2N}{N+2-\mu}}\vr_n^{\frac{4-2N}{N+2-\mu}}
\int_{\mathbb{R}^N}\chi(x)|\nabla \Phi_{\tau_{n},z_{n}}|^{2}dx+
o(1)\\
&=\beta(\Phi_{\tau_{n},z_{n}})+o(1).
\endaligned$$
By choosing subsequences of $\{\tau_{n}\}$ and $\{\vr_n\}$ such that $\frac{\tau_{n_{i}}}{\vr_{n_{i}}}=O(1)$ as $n_{i}\rightarrow\infty$, we may replace  $\{\tau_{n_{i}}\}$ by $\{\vr_{n_{i}}\}$ and relabel $\{\vr_{n_{i}}\}$ by $\{\vr_n\}$. Let
$$
v_{n}(x)=\vr_n^{\frac{N-2}{4}}g_{n}(\sqrt{\vr_n}x+z_{n})
$$
then $v_{n}\rightarrow {U}_{1,0}$ in $D^{1,2}(\mathbb{R}^N)$. Thus from $\displaystyle\int_{\mathbb{R}^N}|\nabla v_{n}(x)|^{2}dx\rightarrow S_{H,L}^{\frac{2N-\mu}{N+2-\mu}}$ and
$$\aligned
\frac{N+2-\mu}{2(2N-\mu)} S_{H,L}^{\frac{2N-\mu}{N+2-\mu}}&\leftarrow \frac{N+2-\mu}{2(2N-\mu)}\int_{\mathbb{R}^N}(|\nabla g_{n}|^{2}+ V(x)|g_{n}|^{2})dx\\
&=\frac{N+2-\mu}{2(2N-\mu)}\int_{\mathbb{R}^N}(|\nabla v_{n}|^{2}+ V(\sqrt{\vr_n}x+z_{n})|v_{n}|^{2})dx,
\endaligned$$
we can conclude that
$$
\lim\limits_{n\rightarrow\infty}\int_{\mathbb{R}^N}V(\sqrt{\vr_n}x+z_{n})|v_{n}|^{2}dx=0.
$$
This implies that $\displaystyle\int_{\mathbb{R}^N}V(\overline{z})|{U}_{1,0}|^{2}dx=0$ and so $V(\overline{z}) = 0$. This means that $\overline{z}\in M$.
Therefore $z_{n}\in M_{\tau}$ for large $n$.

We need only to prove that the sequence $\tau_{n}\rightarrow0$ and $\{z_{n}\}$ is bounded. Since $w_{n}\rightarrow0$ in $D^{1,2}(\mathbb{R}^N)$, we know
$$\aligned
\beta(u_{n})=\beta(\Psi_{\tau_{n},z_{n}})+o(1)
\endaligned$$
Thus, from the choice of $u_n$, we may assume that
\begin{equation}\label{e8}
\beta(\Psi_{\tau_{n},z_{n}})\subset B_{\frac{\rho}{2}}(0).
\end{equation}
If $\tau_{n}\rightarrow\infty$ as $n\rightarrow\infty$, then we know that for each $R>0$ there holds
$$
\lim\limits_{n\rightarrow\infty}\int_{B_{R}(0)}|\nabla \Psi_{\tau_{n},z_{n}}|^{2}dx=0.
$$
Using this fact and the definition of the mapping $\gamma$, we know
\begin{equation}\label{e9}
\aligned
\gamma(\Psi_{\tau_{n},z_{n}})&=S_{H,L}^{\frac{\mu-2N}{N+2-\mu}}\vr_n^{\frac{4-2N}{N+2-\mu}}
\int_{\mathbb{R}^N}|\chi(x)-
\beta(\Psi_{\tau_{n},z_{n}})|
|\nabla \Psi_{\tau_{n},z_{n}}(x)|^{2}dx\\
&\geq S_{H,L}^{\frac{\mu-2N}{N+2-\mu}}\vr_n^{\frac{4-2N}{N+2-\mu}}
\int_{\mathbb{R}^N}|\chi(x)||\nabla \Psi_{\tau_{n},z_{n}}(x)|^{2}dx-|\beta(\Psi_{\tau_{n},z_{n}})|\\
&\geq S_{H,L}^{\frac{\mu-2N}{N+2-\mu}}\vr_n^{\frac{4-2N}{N+2-\mu}}
\int_{\mathbb{R}^N}|\chi(x)||\nabla \Psi_{\tau_{n},z_{n}}(x)|^{2}dx-\frac{\rho}{2}\\
&\geq\rho S_{H,L}^{\frac{\mu-2N}{N+2-\mu}}\vr_n^{\frac{4-2N}{N+2-\mu}}
\int_{\mathbb{R}^N\backslash B_{\rho}(0)}|\nabla \Psi_{\tau_{n},z_{n}}(x)|^{2}dx-\frac{\rho}{2}+o(1)\\
&=\rho S_{H,L}^{\frac{\mu-2N}{N+2-\mu}}\vr_n^{\frac{4-2N}{N+2-\mu}}
\int_{\mathbb{R}^N}|\nabla \Psi_{\tau_{n},z_{n}}(x)|^{2}dx-\frac{\rho}{2}+o(1)\\
&=\frac{\rho}{2}+o(1).
\endaligned
\end{equation}
Thus from the fact that
$$
\gamma(u_{n})=\gamma(\Psi_{\tau_{n},z_{n}})+o(1),
$$
we know
$$
\gamma(u_{n})\geq \frac{\rho}{2}+o(1).
$$
However, since $u_{n}\in \tilde{\mathcal{N}}_{\vr}$ we have
\begin{equation}\label{e10}
\delta_{1}(\vr_n)<\gamma(u_{n})<\delta_{2}(\vr_n)
\end{equation}
where $\delta_{i}(\vr_n)\rightarrow0$, $i = 1, 2$, as $\vr_n\rightarrow0$. This contradicts the estimate \eqref{e9} and
therefore $\{\tau_{n}\}$ is bounded. It remains to show that $\tau_{n}\rightarrow0$ as $n\rightarrow\infty$. On the contrary,
if $\tau_{n}\rightarrow\overline{\tau}>0$ as $n\rightarrow\infty$, then we must have that
$|z_{n}|\rightarrow\infty$ as $n\rightarrow\infty$. Otherwise, up to subsequence, $\Psi_{\tau_{n},z_{n}}$ would converge strongly in $D^{1,2}(\mathbb{R}^N)$ and so
would $u_{n}$. Consequently $I_{\vr}$  possesses nontrivial minimizer on $\mathcal{N}_\vr$
which is impossible by Proposition \ref{NE}. We now observe that for every
$R > 0$, the fact that $\lim\limits_{n\rightarrow\infty}|z_{n}|=\infty$, implies that $\lim\limits_{n\rightarrow\infty}\displaystyle\int_{B_{R}(0)}|\nabla \Psi_{\tau_{n},z_{n}}|^{2}dx=0$.
Consequently one can easily show that the estimate \eqref{e9} must be valid giving the
contradiction with the fact that $u_{n}$ satisfies \eqref{e10}. The proof of the boundedness of the sequence $\{z_n\}$ is
similar and it is omitted.
\end{proof}

\noindent
{\bf Proof of Theorem \ref{3}.}
We fix an $\vr> 0$ small. Then $\Phi_{\delta,z}:[\delta_{1},\delta_{2}]\times M\rightarrow\tilde{\mathcal{N}}_{\vr}$ and by virtue of
\eqref{e1} and Lemma \ref{E4}, $\beta(\tilde{\mathcal{N}}_{\vr})\subset M_{\tau}$. Therefore $\beta(\Phi_{\delta,z}):[\delta_{1},\delta_{2}]\times M\rightarrow[\delta_{1},\delta_{2}]\times M_{\tau}$ and it is easy to check that $\beta(\Phi_{\delta,z})$ is homotopic to
the inclusion map $id:[\delta_{1},\delta_{2}]\times M\rightarrow[\delta_{1},\delta_{2}]\times M_{\tau}$. The functional $I_{\vr}$ satisfies the
$(PS)_{c}$-condition for
$$c\in\left(\vr_n^{\frac{2(2N-\mu)}{N+2-\mu}} \frac{N+2-\mu}{2(2N-\mu)} S_{H,L}^{\frac{2N-\mu}{N+2-\mu}},\vr_n^{\frac{2(2N-\mu)}{N+2-\mu}} \left[\frac{N+2-\mu}{2(2N-\mu)} S_{H,L}^{\frac{2N-\mu}{N+2-\mu}}+h(\vr)\right]\right).$$
Hence by the Lusternik-Schnirelman theory of critical points (see \cite{Wi})
$$
cat(\tilde{\mathcal{N}}_{\vr})\geq cat_{[\delta_{1},\delta_{2}]\times M_{\tau}}([\delta_{1},\delta_{2}]\times M)=cat_{M_{\tau}}M.
$$
$\hfill{} \Box$

\end{document}